%% file: lenscobor_12_2_15.tex
\documentclass[11pt]{amsart}
\usepackage{amsmath}
\usepackage{amssymb,amscd}
\usepackage{textcomp}
\usepackage{graphicx}
\usepackage{color}
\usepackage{perpage} 
\MakePerPage{footnote} 

\numberwithin{equation}{section}

\theoremstyle{plain}
\newtheorem{theorem}{Theorem}[section]

\newtheorem{corollary}[theorem]{Corollary}

\newtheorem{lemma}[theorem]{Lemma}
\newtheorem{prop}[theorem]{Proposition}
\newtheorem{ques}[theorem]{Question}

\theoremstyle{definition}
\newtheorem{defn}[theorem]{Definition}
\newtheorem{remark}[theorem]{Remark}
\newtheorem{example}[theorem]{Example}
\newtheorem{condition}[theorem]{Nonsingularity Condition}
\newtheorem{obsr}[theorem]{Observation}

\def \begineq{\begin{equation}}
\def \endeq{\end{equation}}

\def \bb{\mathbb}

\def \CC{{\bb{C}}}

\def \NN{{\bb{N}}}

\def \RR{{\bb{R}}}

\def \ZZ{{\bb{Z}}}

\def \({\left(}
\def \){\right)}
\def \<{\langle}
\def \>{\rangle}
\def \bar{\overline}

\def \tensor{\otimes}

\def \id{{\rm id}}

\begin{document}

\title{A new construction of lens spaces}

\author[S. Sarkar]{Soumen Sarkar}

\address{Department of Mathematics and Statistics, University of Calgary, Calgary T2N 1N4, Canada}

\email{soumensarkar20@gmail.com}

\author[D. Y. Suh]{Dong Youp Suh}

\address{Department of Mathematical Sciences, KAIST, Daejeon 305701, Republic of Korea}

\email{dysuh@math.kaist.ac.kr}

\thanks{Dong Youp Suh is supported in part by Basic Science Research Program through the National Research Foundation of Korea (NRF) funded by the Ministry of Education (2013R1A1A2007780).}



\abstract  Let $T^n$ be the real $n$-torus group. We give a new definition of lens spaces and study the diffeomorphic classification of lens spaces.
 We show that any $3$-dimensional  lens space $L(p; q)$ is $T^2$-equivariantly cobordant to zero.
We also give some sufficient conditions for higher dimensional lens spaces $L(p; q_1, \ldots, q_n)$ to be $T^{n+1}$-equivariantly cobordant to zero. In 2005, B. Hanke showed that complex equivariant cobordism class of a lens space is trivial. Nevertheless, our proofs are constructive using toric topological arguments.
 
\endabstract

\maketitle

\section{Introduction}\label{intro}
Constructing 3-dimensional lens spaces by surgery on knots was discussed in the paper of Fintushel and Stern \cite{FS}. Inspired by \cite{FS}, one may ask if there is another way to construct any dimensional lens spaces. To answer this question we introduce the hyper-characteristic function on an $n$-simplex to give new construction of a $(2n+1)$-dimensional lens space. This new definition of lens spaces helps to determine the explicit cobordism of many lens spaces.   

Cobordism was first introduced by Lev Pontryagin in his seminal work on manifolds,  \cite{Pon}.
In early 1950's Ren\'{e} Thom \cite{Tho} showed that cobordism groups could be computed through
homotopy theory using the Thom construction, and now we know the oriented, non-oriented and complex
cobordism rings completely. On the other hand, even though there have been a lot of developments,
the equivariant cobordism rings are not determined for any nontrivial groups. Part of the reason
is that the Thom transversality theorem does not hold in equivariant category, and hence the equivariant
cobordism can not be reduced to homotopy theory.

In this article we study equivariant cobordism of lens spaces. In particular, Theorem~\ref{eqcob4}
gives a sufficient condition for a $(2n+1)$-dimensional lens space $L(p;q_1,\ldots,q_n)$ to be a
$T^{n+1}$-equivariant boundary, where $T^{n+1}$ is the rank $n+1$ real torus group. In particular,
Corollary~\ref{main corollary} shows that if any two integers of $p,q_1\ldots,q_n$ are relatively
prime and $q_1, \ldots, q_n$ satisfy the sufficient condition of Theorem \ref{eqcob4}, then  $L(p;q_1,\ldots,q_n)$  is a $T^{n+1}$-equivariant boundary. Moreover, we can explicitely construct a $(2n+2)$-dimensional oriented $T^{n+1}$-manifold which bounds $L(p; q_1, \ldots, q_n)$. 

The main tool to get such results is the theory of quasitoric manifolds, which was introduced by Davis
and Januskiewicz in the pioneering work \cite{DJ}. A quasitoric manifold is a closed $2n$-dimensional manifold $M$ with a
locally standard $T^n$ action whose orbit space has the structure of an $n$-dimensional simple convex
polytope $P$. Each codimension one face $F$ of  $P$, called a {\em facet}, corresponds to codimension
two submanifolds fixed by a circle subgroup $S^1_F$ of $T^n$, and such information is recorded as a
{\em characteristic function}
$$
\lambda\colon \mathcal F(P)\to \mathbb Z^n
$$
defined up to sign. Here $\mathcal F(P)$ denotes the set of facets of $P$. This characteristic function
must satisfy the following {\em nonsingularity condition}.


\begin{condition}
If the intersection $F_1\cap\cdots\cap F_\ell$ of $\ell$ facets of $P$ is an $(n-\ell)$-dimensional face of $P$, then 
the integral vectors $\lambda(F_1), \ldots \lambda(F_\ell)$ form a part of basis of $\mathbb Z^n$ for all $\ell=1,\ldots,n$
\end{condition}

Conversely, for a given $n$-dimensional simple polytope $P$ and a function $\lambda\colon \mathcal F(P)\to \mathbb Z^n$
satisfying the above nonsingularity condition, a quasitoric manifold $M(P,\lambda)$ with $P$ as its orbit space and
$\lambda$ as its characteristic function can be constructed, see 1.5 in \cite{DJ}. Indeed, the associated quasitoric manifold is the quotient space
\begin{equation}\label{DJ-construction}
M(P,\lambda)=T^n\times P/\sim,
\end{equation}
where $\sim$ is the equivalence relation defined by
$$
(t,p)\sim (s,q) \Leftrightarrow p=q\, \textrm{ and} \,\, ts^{-1}\in T_F.$$
Here, $F$ is the unique face of $P$ containing $x$ in its relative interior, and if $F$ is the intersection $F_1\cap \cdots
\cap F_\ell$ of $\ell$ facets, then $T_F$ is the torus subgroup corresponding to the subgroup of $\mathbb Z^n$ generated by
the vectors $\lambda(F_1),\ldots, \lambda(F_\ell)$.

In Section\ref{quotient space}, we modify the definition of characteristic function to a {\em hyper characteristic function}
$$
\xi\colon \mathcal F(\Delta^n)\to \mathbb Z^{n+1}
$$
defined on the set of facets of the $n$-simplex $\Delta^n$. Note that the rank of the target group is $n+1$, instead
of $n$. Then, by  similar construction to (\ref{DJ-construction}), we can construct a $(2n+1)$-dimensional $T^{n+1}$-manifold
$$
L(\Delta^n, \xi)=T^{n+1}\times \Delta^n/\sim
$$
which is called a {\em generalized lens space}.

Let $p>0,\, q_1,\dots,q_n$ be integers such that $p$ and $q_i$ are relatively prime for $i=1,\ldots, n$.
Let $f_0,\ldots, f_n$ be the facets of $\Delta^n$, and consider a hyper characteristic function $\xi$ defined by
\begin{eqnarray*}
\xi(f_0)&=(-q_1,\ldots,-q_n, p),  \,\textrm{ and}\\
\xi(f_i)&=e_i, \,\textrm{ for} \,i=1,\, \ldots, n,
\end{eqnarray*}
where $e_i$ is the standard $i$th basis vector of $\ZZ^{n+1}$. Then the generalized lens space $L(\Delta^n, \xi)$
is exactly the usual $(2n+1)$-dimensional lens space $L(p;q_1,\ldots,q_n)$, which gives an alternative construction of
a $(2n+1)$-dimensional lens space using the technique of toric topology.

Let $L(p;q_1,\ldots,q_n)=L(\Delta^n,\xi)$ be a (2n+1)-dimensional lens space for a hyper characteristic function
$\xi$ defined as above. Now consider the $(n+1)$-simplex $\Delta^{n+1}$, and regard $\Delta^n$ as a facet
of it.  We would like to extend $\xi$ to a {\em rational characteristic function}
$$
\eta\colon \mathcal (\Delta^{n+1})\to \mathbb Z^{n+1},
$$
i.e., $\eta$ satisfy the nonsingularity condition for $\ell=1,\ldots, n-1$ so that
\begin{eqnarray*}
\eta(F_0)&=\xi(f_0)=(-q_1,\ldots,-q_n, p),  \,\textrm{and}\\
\eta(F_i)&=\xi(f_i)=e_i,\, \textrm{ for} \,i=1,\ldots, n,
\end{eqnarray*}
where $F_i$ is the facet containing $f_i$ and not equal to $\Delta^n$.
Then the space $$M=(T^{n+1}\times \Delta^{n+1})/\sim$$ constructed similarly to (\ref{DJ-construction}) is, in general,
an orbifold with singularities occurring at the points corresponding to the vertices of $\Delta^{n+1}$.

Now consider the {\em vertex-cut} $\Delta^{n+1}_V$ of $\Delta^{n+1}$, i.e., cutting off a small disjoint
$(n+1)$-simplex-shaped neighborhoods from each vertex of $\Delta^{n+1}$.  Then
$$
W:=\pi^{-1}(\Delta_V^{n+1})=T^{n+1}\times \Delta^{n+1}_V/\sim\,  \subset M
$$
is a $(2n+2)$-dimensional manifold with boundary, which consists of $(2n+1)$ dimensional lens spaces, and in particular
one of the boundary components is the lens space $L(p;q_1,\ldots,q_n)$. Here $$\pi:(T^{n+1}\times \Delta^{n+1})/\sim\, \to \Delta^{n+1}$$ is the map induced from the projection.  If the numbers $p, q_1, \ldots, q_n$ satisfy certain
number theoretical condition, then we can continue the similar procedure to the other boundary components to show that they are equivariant boundaries, and hence the lens space $L(p;q_1,\ldots,q_n)$ is an equivariant boundary. This is how we get the main results of this article, see Theorem \ref{eqcob3} and \ref{eqcob4}. In 2005, B. Hanke showed that complex equivariant cobordism class of a lens space is trivial, see Theorem 1 in \cite{Han}. Nevertheless, our proofs are combinatorial and constructive. A nice application of this construction can be found in Section 2 of \cite{BSS}.

\section{Some quotient spaces of odd dimensional spheres and lens spaces}\label{quotient space}
In this section, first we recall the known definition of lens spaces. Then we give another description of $(2n+1)$-dimensional lens space $L(p;q_1,\ldots, q_n)$, and find some sufficient condition for two $(2n+1)$-dimensional  lens spaces to be diffeomorphic using toric topological techniques.

Let $p>0, q_1,\ldots, q_n$ be integers such that $p$ and $q_i$ are relatively prime for all $i=1,\ldots, n$.
The {\em  $(2n+1)$-dimensional lens space}  $L(p;q_1,\ldots, q_n)$  is the orbit space $S^{2n+1}/\mathbb Z_p$ where $\mathbb Z_p$ action on
$S^{2n+1}$  is defined by
$$\theta\colon \mathbb Z_p\times S^{2n+1}\to S^{2n+1}$$
$$(k, (z_1,\ldots, z_n))\mapsto (e^{2kq_1\pi \sqrt{-1}/p}z_1, \ldots, e^{2kq_n\pi \sqrt{-1}/p}z_n, e^{2k\pi \sqrt{-1}/p}z_{n+1})$$
where $S^{2n+1}=\big\{(z_1,\ldots, z_{n+1})\in \mathbb C^{n+1} ~~\big|~~ |z_1|^2+\cdots + |z_{n+1}|^2=1 \big\}$.

There is an alternative description of a $3$-dimensional lens space $L(p; q)$  as the result of gluing
two solid tori via an appropriate homeomorphism of
their boundaries, see  \cite{OR}. Moreover it is shown in
(\cite{Br2}) that $L(p; q)$ is homeomorphic
to $L(p; r)$ if and only if $$r \equiv \pm q ~ (\rm{mod} ~p) ~\mbox{or} ~ qr \equiv \pm 1 ~ (mod ~p).$$

Now we give new construction of lens spaces. 
Let $\Delta^n=V_0V_1\cdots V_n$ be an $n$-dimensional simplex with vertices $V_0,\ldots, V_n$.
Let $F_i$ be the facet of $\Delta^n$ which does not contain $V_i$, and let $\mathcal F(\Delta^n)$ donate the set
$\{F_0,\ldots, F_n\}$ of facets of $\Delta^n$.

\begin{defn}
A function $\xi: \mathcal{F}(\Delta^n) \to \ZZ^{n+1}$ is called a $hyper$ $characteristic$ function on
$\Delta^n$ if whenever $F_{i_1} \cap F_{i_2} \cap \ldots \cap F_{i_\ell}$ is nonempty, the submodule generated by
$\{\xi(F_{i_1}), \ldots, \xi(F_{i_\ell})\}$ is a direct summand of $\ZZ^{n+1}$ of rank $\ell$. For notational convenience,
let us denote $\xi(F_i)$ by $\xi_i$ for $i=0, \ldots, n$.
\end{defn}

\begin{example}
Some hyper characteristic functions for triangles are given in Figure \ref{len901}.
One can show that the submodules generated by $\{(0,2,3), (4,1,0)\}$, $\{(4,1,0), (3,2,4)\}$ and $\{(3,2,4), (0,2,3)\}$ are direct summands of $\ZZ^3$. We can check that the function given in Figure \ref{len901} $(b)$ is a hyper characteristic function if and only if $p$ is relatively prime to each $q_1$ and $q_2$.

\begin{figure}[ht]
\centerline{
\scalebox{0.80}{
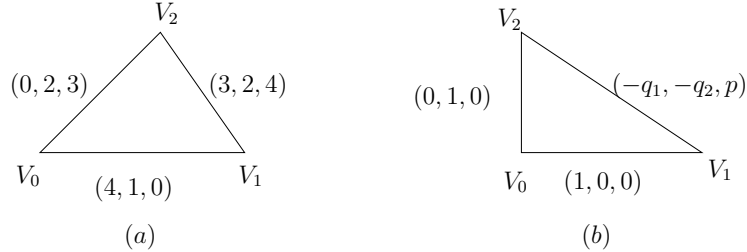
 }
 }
\caption {Some hyper characteristic functions of triangle.}
\label{len901}
\end{figure}
\end{example}

Of course a hyper characteristic function can be defined for more general simple convex polytope $P$,
but we only consider the simplex case here.

Let $ F $ be a face of $ \Delta^n $ of codimension $ \ell $. Then $ F $ is the intersection
of a unique collection of $\ell$ facets $F_{i_1}, F_{i_2}, \ldots, F_{i_\ell}$ of $ \Delta^n $.
Let $ T_F $ be the torus subgroup of $ T^{n+1}$ corresponding to the submodule generated by
$ \{\xi_{i_1}, \xi_{i_2}, \ldots, \xi_{i_\ell} \}$ of $ \ZZ^{n+1} $. Fix $T_{\Delta^n} = 1$. We define an equivalence relation
$\sim$ on the product $ T^{n+1} \times \Delta^n $ as follows.
\begin{equation}\label{eqich}
(t, x) \sim (u, y) \quad \mbox{if and only if} \quad  x = y ~ \mbox{and} ~ tu^{-1} \in T_{F}
\end{equation}
where $ F \subset \Delta^n$ is the unique face containing $ x $ in its relative interior.
We denote the quotient space $ (T^{n+1} \times \Delta^n)/ \sim $ by $ L(\Delta^n, \xi)$.
Let $\ZZ(\Delta^n, \xi) \subset \ZZ^{n+1}$ be the submodule generated by $\xi_0, \ldots, \xi_{n}$,
and let  $$\ZZ_{\xi}(\Delta^n) := \ZZ^{n+1}/ \ZZ(\Delta^n, \xi).$$
Note that the rank of $\ZZ(\Delta^n, \xi)=n$ or $n+1$, since $\xi$ is a hyper characteristic function.

\begin{prop}\label{genlensp}
 The quotient space $ L(\Delta^n, \xi) $ is a $(2n+1)$-dimensional topological
manifold  with the natural effective action of $T^{n+1}$ induced from the group operation on the first factor of $T^{n+1}\times \Delta^n$. Furthermore,
\begin{enumerate}
\item if the rank of $\ZZ(\Delta^n, \xi)$ is $n$ then $L(\Delta^n,\xi)$ is homeomorphic to 
$ S^1\times\mathbb CP^n$, and
\item if the rank of $\ZZ(\Delta^n, \xi)$ is $n+1$ then $L(\Delta^n,\xi)$ is homeomorphic to 
$ S^{2n+1}/\ZZ_{\xi}(\Delta^n)$.
      In particular, if $\{\xi_0,\ldots, \xi_n\}$ form a basis for $\ZZ^{n+1}$ then  $L(\Delta^n,\xi)$
is homeomorphic to $S^{2n+1}$.
\end{enumerate}
\end{prop}

\begin{proof}
Let $U_{i}= \Delta^n - F_{i}$ for $i=0,1, \ldots, n$. Then $U_i$ is diffeomorphic to
$$\RR^n_{\geq 0} =\{(x_1, \ldots, x_n) \in \RR^n\mid x_i\geq 0\, \textrm{for}\, i=1,\ldots n\}.$$
Let $f^i: U_i \to \RR^n_{\geq 0}$ be a diffeomorphism.
Let $S^1_{\xi_i}$ be the circle subgroup of $T^{n+1}$ determined by the vector $\xi_i$ for $i=0, \ldots, n$.
So $T^{n+1}= S^1_i \times S^1_{\xi_0} \times \ldots \times \widehat{S^1_{\xi_i}} \times \ldots \times S^1_{\xi_n}$
for some circle subgroup $S^1_i$ determined by some primitive vector $v_i \in \ZZ^{n+1}$, where $~\widehat{}~$
represents the omission of the circle $S^1_{\xi_i}$.  Let $\{e_1, ..., e_{n+1}\}$ be the standard basis of
$\ZZ^{n+1}$ over $\ZZ$. Consider the diffeomorphism $$g^i : T^{n+1} \to T^{n+1}$$
defined by $g^i(v_i)=e_{i+1}$ and $g^i(\xi_j)= e_{j+1}$ for $j \neq i$. So the diffeomorphism $$g^i \times f^i: T^{n+1} \times U_i \to T^{n+1} \times \RR^n_{\geq 0}$$ induces an weakly-equivariant homeomorphism from
$(T^{n+1} \times U_i)/\sim$ to $$(T^{n+1} \times \RR^n_{\geq 0})/ \sim_e= S^1_{e_{i+1}} \times ((S^1_{e_1}
\times \ldots \times \widehat{S^1_{e_{i+1}}} \times \ldots \times S^1_{e_{n+1}})/\sim_e) \cong S^1_{e_{i+1}} \times \CC^{n}$$ where $\sim_e$ is the relation $\sim$ of Lemma $1.6$ in \cite{DJ}. Hence $ L(\Delta^n, \xi)$
is covered by the $(2n+1)$-dimensional open sets $(T^{n+1} \times U_i)/\sim$. Thus $ L(\Delta^n, \xi)$ is a $(2n+1)$-dimensional topological manifold. The $T^{n+1}$-action on the first factor of $T^{n+1} \times \Delta^n$ induces an effective $T^{n+1}$-action on $ L(\Delta^n, \xi)$.

Suppose the rank of $\ZZ(\Delta^n, \xi)$ is $n$. Since $\xi$ is a hyper characteristic function, the
submodule generated by $\{\xi_0, \ldots, \xi_{n-1}\}$ is a direct summand of rank $n$. Then the vector $\xi_n $
is contained in the subgroup $<\xi_0, \ldots, \xi_{n-1}>$ and  there exists $v \in  \ZZ^{n+1}$ such that
$\{\xi_0, \ldots, \xi_{n-1}, v\}$ is a basis of $\ZZ^{n+1}$. So by considering some automorphism on $\ZZ^{n+1}$,
if necessary, we may regard $\xi$ as a characteristic function of $\Delta^n$ in the sense of \cite{DJ}.
Let $S^1_{v}$ and $T_v$ be the subgroups of $T^{n+1}$ determined by $\{v\}$ and $\{\xi_0, \ldots, \xi_{n-1}\}$. So $T^{n+1} \cong S^1_v \oplus T_v$. Then $$(T^{n+1} \times \Delta^n)/\sim ~\cong S^1_{v} \times (T_{v} \times \Delta^n)/\sim ~\cong S^1_v \times \mathbb{CP}^n.$$

Now, assume that the rank of $\ZZ(\Delta^n, \xi)$ is $n+1$. 
Let $\{e_1, ..., e_{n+1}\}$ be the standard basis of $\ZZ^{n+1}$ over $\ZZ$. Define
$$\xi^s(F_i)= e_{i+1}= \xi^{s}_i \quad \mbox{for} \quad i=0, \ldots, n$$ called the standard hyper characteristic function of
$\Delta^n$. Consider the standard action of $T^{n+1}$ on $\CC^{n+1}$. The orbit
map $$\pi_s: \CC^{n+1} \to \RR^{n+1}_{\geq 0}$$ of this action is given by $(z_1, \ldots, z_{n+1})
\to (|z_1|, \ldots, |z_{n+1}|)$. Let $$H=\big\{(x_1, \ldots, x_{n+1}) \in \RR^{n+1}_{\geq} ~~\big|~~
x_1+ \cdots + x_{n+1}=1 \big\}.$$ Then $H$ is diffeomorphic as manifold with corners to $\Delta^n$.
Facets of $H$ are $$H_i=\{(x_1, \ldots, x_{n+1}) \in H ~~|~~ x_i =0 \}$$ for $i= 1, \ldots, n+1$.
The isotropy subgroup of $\pi_s^{-1}(H_i)$ is the $i$th circle subgroup of $T^{n+1}$.
So we get a hyper characteristic function on $H$ which is nothing but the standard one.
Hence it is clear that $$S^{2n+1}= \pi_s^{-1}(H) \cong (T^{n+1} \times H)/\sim_s \cong
(T^{n+1} \times \Delta^n)/ \sim_s = L(\Delta^n, \xi^s),$$ where $\sim_s$
is the equivalence relation $ \sim$ defined in (\ref{eqich}) corresponding to the standard hyper
characteristic function $\xi^s$. Consider the map $$\beta: \ZZ^{n+1} \to \ZZ^{n+1} \quad \mbox{defined by} \quad \beta(e_i=\xi^s_{i-1}) = \xi_{i-1}$$ for $i=1, ..., n+1$. Let $\ZZ_{\RR}^{n+1}= \ZZ^{n+1}
\tensor_{\ZZ} \RR$. Since the rank of $\ZZ(\Delta^n, \xi)$ is $n+1$, the map $\beta$
induces a surjective homomorphism
$$\bar{\beta}: T^{n+1} \cong \ZZ^{n+1}_{\RR}/\mbox{Im}(\beta) \to \ZZ^{n+1}_{\RR}/\ZZ^{n+1} \cong T^{n+1}$$
defined by $v+ \mbox{Im}(\beta) \to v+ \ZZ^{n+1}$.The kernel of $\bar \beta$ is $\ZZ^{n+1}/\mbox{Im}(\beta)
= \ZZ_{\xi}(\Delta^n)$, a finite subgroup of $T^{n+1}$. From the definition of $\sim_s$ and $\sim$
We get the following commutative diagram
$$
\begin{CD}
T^{n+1} \times \Delta^n @>\bar{\beta} \times \id>> T^{n+1} \times \Delta^n\\
@VVV  @VVV \\
S^{2n+1} = (T^{n+1} \times \Delta^n)/\sim_s @>f_{\beta}>> ((T^{n+1} \times \Delta)/ \sim = L(\Delta^n, \xi)
\end{CD}
$$
where $f_{\beta}$ is defined by $f_{\beta}([t,x]^{\sim_s})=[\bar{\beta}(t), x]^{\sim}$ on the equivalence classes. So $f_{\beta}$ is a continuous surjective map. The finite group $\ZZ_{\xi}(\Delta^n)$ has a natural free and smooth action on $T^{n+1}$ induced by the group operation. This induces a smooth action of $\ZZ_{\xi}(\Delta^n)$ on $S^{2n+1}$. Since $\bar {\beta}$ is a covering homomorphism with the finite covering group
$\ZZ_{\xi}(\Delta^n)$, the map $f_{\beta}$ induces a bijective continuous map between compact Hausdorff spaces $S^{2n+1}/\ZZ_{\xi}(\Delta^n)$ and $L(\Delta^n, \xi)$. Therefore $L(\Delta^n, \xi)$ is homeomorphic to the quotient space $S^{2n+1}/\ZZ_{\xi}(\Delta^n)$.
\end{proof}

In the case when rank of $\ZZ(\Delta^n, \xi)$ is $n+1$,  we call the space $L(\Delta^n, \xi)$
a {\em generalized~ lens~ space} corresponding to the hyper characteristic function $\xi$ on $\Delta^n$.
We call  $(\Delta^n, \xi)$  a {\em combinatorial model} for the generalized lens space.

\begin{remark}
 If $\{\xi_0, \ldots, \xi_n\}$ is a basis of $\ZZ^{n+1}$ over $\ZZ$ then the manifold $L(\Delta^n, \xi)$
is the sphere $S^{2n+1}$, the moment angle manifold of $\Delta^n$, where the natural action of $T^{n+1}$
may differ from the standard action on $S^{2n+1}$ by an automorphism of $T^{n+1}$.
\end{remark}

We now consider generalized lens spaces with a particular type of hyper characteristic functions.
Let $p>0,q_1, ..., q_n$ be integers such that $p$ is relatively prime to each $q_i$ for $i=1, ..., n$. Define a function $$\xi: \mathcal{F}(\Delta^n) \to \ZZ^{n+1}$$ by $\xi(F_i) = e_i$ for $i=1, ..., n$ and
$\xi(F_0)=(-q_1, -q_2, ..., -q_n, p)$. So $\xi$ is a hyper characteristic function of $\Delta^n$. The rank
of the submodule generated by $\{\xi(F_i)\mid i=0, \ldots, n\}$ is $n+1$. In this case the surjective homomorphism
$\bar{\beta}: T^{n+1} \to T^{n+1}$ induced by $\xi$ is given by,
\begin{equation}\label{cl1}
\bar{\beta} : (t_1, \ldots, t_{n+1}) \to (t_1 t_{n+1}^{-q_1}, \ldots, t_n t_{n+1}^{-q_n}, t_{n+1}^p).
\end{equation}
So $\ZZ_{\xi}(\Delta^n) \cong
\{(t_, \ldots, t_{n+1}) \in T^{n+1}\mid t_{n+1}^{-q_i}t_i=1 ~\mbox{for} ~i=1, \ldots, n~ \mbox{and}~t_{n+1}^p=1\}.$
If $\omega$ be the $p$-th root of unity, then $$\ZZ_{\xi}(\Delta^n) \cong \{(\omega^{q_1}, \ldots, \omega^{q_n}, \omega) \in T^{n+1} \mid w^p =1\} \cong \ZZ_p.$$ The $\ZZ_{\xi}(\Delta^n)$-action on $S^{2n+1}$ induced by the group
operation on $T^{n+1}$ is nothing but the following: 
$$(\omega^{q_1}, \ldots, \omega^{q_n}, \omega) \times
(z_1, z_2, \ldots, z_{n+1}) \to (\omega^{q_1}z_1, \ldots, \omega^{q_n}z_{n}, \omega z_{n+1})$$ where $(z_1, \ldots,
z_{n+1}) \in S^{2n+1}$. Hence $L(\Delta^n, \xi)$ is a  usual $(2n+1)$-dimensional lens space $ L(p; q_1, \ldots, q_n)$. In this case, since the $\ZZ_{\xi}(\Delta^n)$-action on $S^{2n+1}$ is free, the space $L(\Delta^n, \xi)$ is smooth. Moreover, the word `homeomorphic' in Proposition \ref{genlensp} can be replaced by `diffeomorphic'.

We now discuss some classification results of $(2n+1)$-dimensional lens spaces.
For an automorphism $\delta$ on $T^{n+1}$ and a combinatorial model $(\Delta^n, \xi)$ of a generalized lens space $L$,
the $\delta$-{\em translation} of $(\Delta^n, \xi)$ is the combinatorial model $(\Delta^n,\delta(\xi))$
where $$\delta(\xi):\mathcal F(\Delta^n)\to \mathbb Z^{n+1}$$ is the hyper characteristic function
such that $\delta(\xi)(F_i)$ is the vector in $\ZZ^{n+1}$ up to sign
determined by the circle subgroup $\delta(T^1_{\xi_i})$. Here $T^1_{\xi_i}$ is the  circle
subgroup of $T^{n+1}$  determined by the vector $\xi_i$.

\begin{defn}
A diffeomorphism  $g : {L_1}\to  {L_2}$ between two $T^{n+1}$-manifolds $L_1$ and $L_2$  is  $\delta$-{\em equivariant} (or {\em weekly-equivariant }) for   an automorphism $\delta$ of $T^{n+1}$  if $g$ satisfies $g(t\cdot x)= \delta(t)\cdot g(x)$ for all $(t, x) \in T^{n+1} \times L_1$.
\end{defn}
The arguments of the proof of the following lemma is similar to classification of quasitoric manifolds,
see the proof of Proposition $1.8$ of \cite{DJ}.

\begin{lemma}\label{lenscla} Let $L_1$ and $L_2$ be two lens spaces with combinatorial models $(\Delta^n, \xi^1)$ and $(\Delta^n, \xi^2)$ respectively.
Then  $L_1$ and $L_2$ are $\delta$-equivariantly diffeomorphic if and only if $(\Delta^n, \xi^2)$ is a $\delta$-translation of $(\Delta^n, \xi^1)$.
\end{lemma}

The following lemma gives a classification of lens spaces $L(p;q_1,\ldots,q_n)$ up to diffeomorphisms.  
From the integers $q_1,\ldots, q_n$ we obtain the integers $r_1, \ldots, r_n$ as follows. Let ${\bf q}:=
(-q_1, \ldots, -q_n)^t$.
Choose any $B \in SL(n, \ZZ)$, and let ${\bf a}=(a_1, \ldots, a_n)^t:= B {\bf q}$. Then consider a vector 
${\bf a}^{\prime}=(a'_1, \ldots,a'_n)^t$ such that ${\bf a}^{\prime} \equiv
{\bf a}$ (mod $p$). Now let ${\bf q}^{\prime}:=(q_1^{\prime}, \ldots, q_n^{\prime})^t= - B^{-1} {\bf a}^{\prime}$,
and choose ${\bf r}:= (r_1, \ldots, r_n)^t$ to be  ${\bf r}\equiv {\bf q}^{\prime}$ (mod $p$).

\begin{lemma}\label{lenscla2}
Let $p ~(> 0), q_1, \ldots, q_n$ be integers such that $p$ is relatively prime to each $q_i$. Let $r_i$ for $i=1,\ldots, n$ be
the integers obtained as above. Then two lens spaces $L(p; q_1, \ldots, q_n)$ and $L(p; r_1, \ldots, r_n)$ are diffeomorphic.
\end{lemma}

\begin{proof}
Let $\xi$ be a hyper characteristic function of $\Delta^n$ defined by $\xi_0=(-q_1, \ldots,-q_{n}, p)$
and $\xi_i = e_i$ for $i=1, \ldots, n$. So $L(\Delta^n, \xi)= L(p; q_1, \ldots, q_{n})$.
Let $B=(b_{ij}) \in SL(n, \ZZ)$ and ${\bf a} =(a_1, \ldots, a_n)^t = B {\bf q}$.
Consider the map $\delta: \ZZ^{n+1} \to \ZZ^{n+1}$ represented by the matrix
\[
\begin{pmatrix}
B & 0\\
0 & 1
\end{pmatrix}.
\]
Then $\delta$ induces an automorphism of $T^{n+1}$, which is denoted by the same $\delta$.
Consider the hyper characteristic function $\delta(\xi)$, defined by $\delta(\xi)_i= \delta(\xi_i)$ for
$i=0, \ldots, n$. Then $\delta(\xi)$ induces a surjective homomorphism $T^{n+1} \to T^{n+1}$, defined by
$$(t_1, \ldots, t_n, t_{n+1}) \to (t_1^{b_{11}}\cdots t_{n}^{b_{1n}}t_{n+1}^{a_1}, \ldots,
t_1^{b_{n1}}\cdots t_n^{b_{nn}}t_{n+1}^{a_n}, t_{n+1}^p).$$
If ${\bf a}^{\prime}= (a_1^{\prime}, \ldots, a_n^{\prime})^t \equiv {\bf a}$  ($\mod p$), then 
The kernel of this map is given by
$$\ZZ_{\delta(\xi)}(\Delta^n) \cong \{(t_1, \ldots, t_n, t_{n+1}) \in T^{n+1} \mid
t_1^{b_{i1}} \cdots t_{n}^{b_{in}} t_{n+1}^{a_i^{\prime}}=1 ~\mbox{for} ~i=1, \ldots, n ~ \mbox{and}~t_{n+1}^p=1\}.$$
Considering the Lie algebra of $\ZZ_{\delta(\xi)}(\Delta^n)$, we need to find
$(x_1, \ldots, x_{n+1}) \in \ZZ^{n+1}$ such that
\begin{equation}
 b_{i1}x_1+ \cdots + b_{i_n}x_n + a_i^{\prime} x_{n+1}=0~ \mbox{and}~ p x_{n+1}=0 ~ \mbox{for}~
i=1, \ldots, n.
\end{equation}
Namely, $$B \begin{pmatrix}
x_1\\
\vdots \\
x_n
\end{pmatrix}
= - \begin{pmatrix}
a_1^{\prime}\\
\vdots \\
a_n^{\prime}
\end{pmatrix}x_{n+1}. $$
Then
$$\begin{pmatrix}
x_1\\
\vdots \\
x_n
\end{pmatrix}
=-B^{-1}\begin{pmatrix}
a_1^{\prime}\\
\vdots \\
a_n^{\prime}
\end{pmatrix}x_{n+1}
=\begin{pmatrix}
q_1^{\prime}\\
\vdots \\
q_n^{\prime}
\end{pmatrix}x_{n+1}.$$
So $x_i=q_i^{\prime}x_{n+1}$, i.e., $t_i=t_{n+1}^{{q_i}^{\prime}}$. Since $r_i \equiv q_i^{\prime}$
(mod $p$), we have $$\ZZ_{\delta(\xi)}(\Delta^n) \cong \{(\alpha^{r_1}, \ldots, \alpha^{r_n}, \alpha)
\in T^{n+1}: \alpha^p=1\}.$$

The automorphism $\delta$ induces a $\delta$-equivariant diffeomorphisms $$\bar{\delta}: S^{2n+1} \cong L(\Delta^n, \xi^s) \to L(\Delta^n, \delta(\xi^s)) \cong S^{2n+1}$$
and $$\bar{\delta(\xi)}: L(\Delta^n, \xi) \to L(\Delta^n, \delta(\xi)),$$ by Lemma
\ref{lenscla}. Here $\xi^s$ is the standard hyper characteristic function defined in the proof of Proposition~\ref{genlensp}.
Clearly, the following diagram  is commutative where vertical arrows are orbit maps of the actions of
$\ZZ_{\xi}(\Delta^n)$ and $\ZZ_{\delta(\xi)}(\Delta^n)$ on $S^{2n+1}$.
\begin{equation}
\begin{CD}
S^{2n+1} @>{\bar{\delta}}>> S^{2n+1}\\
@VVV  @VVV\\
 L(\Delta^n, \xi) @>\bar{\delta(\xi)} >> L(\Delta^n, \delta(\xi)).
\end{CD}
\end{equation}
Since $\ZZ_{\xi}(\Delta^n)$ acts freely on $S^{2n+1}$, $\ZZ_{\delta(\xi)(\Delta^n)}$
acts freely on $S^{2n+1}$. Therefore we have $L(\Delta^n, \delta(\xi)) \cong L(p; r_1, \ldots, r_n)$. 
Hence $L(p; q_1, \ldots, q_{n})$ is diffeomorphic to $L(p; r_1, \ldots, r_{n})$.
\end{proof}

\begin{example}
Consider the hyper characteristic functions of a triangle $\Delta^2$ given in Figure \ref{len904}.
The hyper characteristic function $\xi^2$ in $(b)$ is the $\delta$-translation of the
hyper characteristic function $\xi^1$ in $(a)$, where $\delta$ is represented by
\[ \begin{pmatrix} B & 0\\ 0 & 1 \end{pmatrix} ~\mbox{for}~ B = \begin{pmatrix} 3 & 5\\
2 & 3 \end{pmatrix}. \]
So
\[ {\bf a} = \begin{pmatrix} a_1\\ a_2 \end{pmatrix} =\begin{pmatrix} 3 & 5\\ 2 & 3 \end{pmatrix}
\begin{pmatrix} 5\\ 7 \end{pmatrix} = \begin{pmatrix} 50\\ 31 \end{pmatrix} \equiv \begin{pmatrix}
2\\ 3 \end{pmatrix} = \begin{pmatrix} a_1^{\prime}\\ a_2^{\prime} \end{pmatrix} = {\bf a}^{\prime}
 ~(\mbox{mod}~ 8) \]
Then \[ {\bf q}^{\prime} = \begin{pmatrix} q_1^{\prime}\\ q_2^{\prime} \end{pmatrix}= -B^{-1} {\bf a}^{\prime}
=-\begin{pmatrix} 3 & -5\\ -2 & 3 \end{pmatrix} \begin{pmatrix} 2\\ 3 \end{pmatrix} = \begin{pmatrix}
9\\ -5 \end{pmatrix} \equiv \begin{pmatrix} 1\\ 3 \end{pmatrix} = \begin{pmatrix}
r_1\\ r_2 \end{pmatrix}~ (\mbox{mod} ~8). \]
So $$\ZZ_{\xi^1}(\Delta^2)= \{(t^{-5}, t^{-7}, t) \in T^3 \mid t^8=1\} \quad \mbox{and} \quad
\ZZ_{\xi^2}(\Delta^2)= \{(t^{1}, t^{3}, t) \in T^3\mid t^8=1\}.$$
Hence by Lemma \ref{lenscla2}, the lens spaces $L(\Delta^2, \xi^1)=L(8; -5, -7)$
and $L(\Delta^2, \xi^2)=L(8; 1, 3)$ are diffeomorphic.
\begin{figure}[ht]
\centerline{
\scalebox{0.80}{
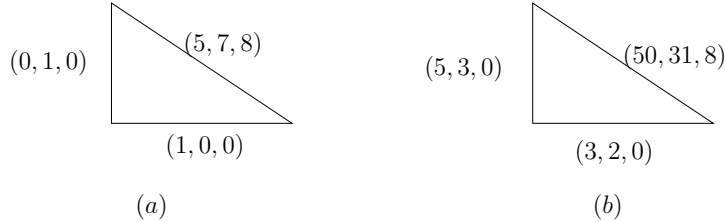
 }
 }
\caption {Hyper characteristic functions $\xi^1$ and $\xi^2$ of a triangle $\Delta^2 $.}
\label{len904}
\end{figure}
\end{example}

\section{Manifolds with lens spaces boundary}\label{lensbdd}
In this section we construct oriented $T^n$-manifold with boundary where the boundary is a disjoint union of lens spaces. Similar construction can be found in Section 4 of \cite{Sar2}. Let $Q$ be an $n$-dimensional simple convex  polytope in $\RR^n$ with facets $F_1, \ldots, F_{m}$ and vertices $V_1, \ldots, V_{k}$. Let $\mathcal{F}(Q)$ denote the set $\{F_1, \ldots, F_{m}\}$ of the facets of $Q$.

\begin{defn}\label{diso}
A function $\eta : \mathcal{F}(Q) \to \ZZ^{n}$ is called a {\em rational characteristic function} on $Q$
if the set of vectors $\{\eta(F_{i_1}), \ldots, \eta(F_{i_{\ell}})\}$ form a part of a basis of $\ZZ^{n}$
whenever the intersection of the facets $\{F_{i_1}, ..., F_{i_\ell}\}$ is an $(n-\ell)$-dimensional face of $Q$, where $n-\ell >0$. The vectors $\eta_i := \eta(F_i)$ for $i=1, \ldots, m$ are called {\em rational characteristic vectors}.
\end{defn}

Note that the definition of a rational characteristic function is same as that of  a characteristic function of a quasitoric manifold in \cite{DJ} except  when $\ell=0$, i.e., when $F_{i_1} \cap \ldots \cap F_{i_n}$ is a vertex of the polytope $Q$.

\begin{example}
Clearly the function given in the Figure \ref{len902} $(a)$ is a rational characteristic
function of a rectangle. We can check that the function given in Figure \ref{len902} $(b)$
is a rational characteristic function of the tetrahedron if and only if any two integers
of $\{q_1, q_2, p\}$ are relatively prime.
\begin{figure}[ht]
\centerline{
\scalebox{0.80}{
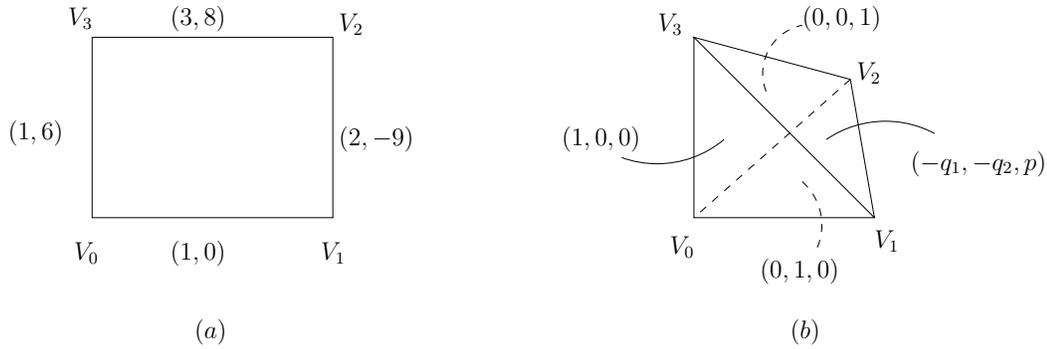
 }
 }
\caption {Some rational characteristic functions of rectangle and tetrahedron.}
\label{len902}
\end{figure}
\end{example}

Let  $\eta : \mathcal{F}(Q) \to \ZZ^{n}$ be a rational characteristic function of $Q$. Let $ F $ be a face of $ Q $ of codimension $ \ell $ with $0 < \ell < n$. Since $Q$ is a simple polytope, $ F $ is the intersection of a unique collection of $ \ell $ many facets $F_{i_1}, F_{i_2}, \ldots, F_{i_\ell}$ of $ Q $. Let $ T_F $ be the torus subgroup of $ T^{n}$ corresponding to the submodule generated by $ \eta_{i_1}, \eta_{i_2}, \ldots, \eta_{i_\ell} $ in $ \ZZ^{n} $. We assume $T_Q=1$ and $T_{V_i}=T^n$ for each vertex $V_i$ of $Q$. We define an equivalence relation $\sim_b$ on the product $ T^{n} \times Q $ as follows:
\begin{equation}\label{equieta}
 (t, x) \sim_b (u, y) \quad \mbox{if and only if} \quad x = y ~ \mbox{and} ~ tu^{-1} \in T_{F}
\end{equation}
where $ F \subset Q$ is the unique face containing $ x $ in its relative interior. We denote the quotient
space $ (T^{n} \times Q)/ \sim_b $ by $ X(Q, \eta) $ and the equivalence class of $(t, x)$ by $[t, x]^{\sim_b}$.
The space $ X(Q, \eta) $ is a manifold if and only if the set of vectors $\{\eta(F_{i_1}), \ldots, \eta(F_{i_{n}})\}$
form a basis of $\ZZ^{n}$ whenever $F_{i_1} \cap \cdots \cap F_{i_n}$ is a vertex of $Q$. In this case, the space
$X(Q, \eta)$ is called a {\em quasitoric manifold} which was introduced by M. Davis and T. Januszkiewicz in \cite{DJ}.
The space $X(Q, \eta)$ is a quasitoric orbifold if the rank of the submodule generated by $\{\eta(F_{i_1}), \ldots,
\eta(F_{i_{n}})\}$ is $n$ whenever $F_{i_1} \cap \cdots \cap F_{i_n}$ is a vertex of $Q$, see the Section $2$ in
\cite{PS}. Let $$\pi: X(Q, \eta) \to Q$$ be the projection map defined by $\pi([t,x]^{\sim_b})=x$.

We can construct $T^n$-manifold with boundary from the orbifold $X(Q, \eta)$ as follows. Cut off a neighborhood of each
vertex $ V_i, i= 1, 2, \ldots, k $ of the simple polytope $Q$ by an affine hyperplane $ H_i, i= 1, 2, \ldots, k$
in $\RR^n$ such that $H_i \cap H_j \cap Q$ are empty sets for $i \neq j$. We call this operation the {\em vertex
cut} of $Q$. Then the remaining subset of the convex polytope $Q$ is an $n$-dimensional simple convex polytope,
denoted by $ Q_V $. Then  the facet $$\Delta^{n-1}_{i}:= Q \cap H_i (= Q_V \cap H_i)$$  of $Q_V$ is  
an $(n-1)$-dimensional simplex for each $i= 1, 2, \ldots, k $.
We restrict the equivalence relation
$\sim_b$ in  (\ref{equieta}) to $T^{n} \times Q_V$, and consider the quotient space
\begin{equation} \label{bounded manifold}
 W(Q_V, \eta) = (T^{n} \times
Q_V )/\sim_b ~ \subset X(Q, \eta).
\end{equation}
The natural action of $ T^{n} $ on $ W(Q_V, \eta) $
is induced by the group operation in $ T^{n}$.

\begin{lemma}\label{lpbdd}
Let $\eta$ be a rational characteristic function on a simple polytope $Q$,  and let $Q_V$ be the vertex cut of $Q$. Then  $ W(Q_V, \eta) $ of \eqref{bounded manifold} is an oriented $2n$-dimensional $T^n$-manifold with boundary,
 whose boundary is a disjoint union of $(2n-1)$-dimensional generalized lens spaces.
\end{lemma}

\begin{proof}
We consider $C_j=\{F : F ~ \mbox{is a face of} ~ Q_V ~\mbox{and}~ F \cap \Delta^{n-1}_j = \emptyset\}$
and $$U_j = Q_V - \cup_{F \in C_j} F$$ for $j=1, \ldots, k$. Since $Q_V$ is a simple polytope,
$U_j$ is homeomorphic as manifold with corners to $\Delta^{n-1}_j \times [0, 1)$ and $Q_V = \cup_{j=1}^k U_{j}$. Let $$ f_j \colon U_j \to \Delta^{n-1}_j \times [0,1)$$ be a homeomorphism. Notice that the facets of $\Delta^{n-1}_j$ are $\{\Delta^{n-1}_j \cap F_{j_1}, \ldots, \Delta^{n-1}_j \cap F_{j_n}\}$ for some facets $F_{j_1}, \ldots, F_{j_n} \in \{F_1, \ldots, F_m\}$ such that $V_j = F_{j_1} \cap \cdots \cap F_{j_n}$.
The restriction of the rational characteristic function $\eta$ on the facets of $\Delta^{n-1}_j$ is given by
$$\xi^j(\Delta^{n-1}_j \cap F_{j_i})= \eta_{j_i} \quad \mbox{for} \quad i = 1, \ldots, n.$$
By definition of the rational characteristic function $\eta$, we can see that $\xi^j$ is a hyper characteristic function on $\Delta^{n-1}_j$. So by Section \ref{quotient space}, the space $L(\Delta^{n-1}_j, \xi^j) =(T^{n} \times \Delta^{n-1}_j)/\sim_b$ is a $(2n-1)$-dimensional generalized lens spaces for $j=1, \ldots, k$. From the equivalence relation $\sim_b$ in (\ref{equieta}), we have the following commutative diagram where lower horizontal maps are homeomorphisms.
$$
\begin{CD}
T^n \times U_j @>Id \times f_i>> T^n \times \Delta^{n-1}_j \times [0, 1) @.\\
@VVV  @VVV @.\\
(T^{n} \times U_j)/\sim_b @>h_i>> ((T^{n} \times \Delta^{n-1}_j)/ \sim_b)\times [0,1) @>\cong>> L(\Delta^{n-1}_j, \xi^j) \times [0, 1).
\end{CD}
$$
So $$W(Q_V, \eta)~ = ~ \bigcup_{j=1}^k (T^{n} \times U_j)/\sim_b ~\cong ~\bigcup_{j=1}^k (L(\Delta^{n-1}_j, \xi^j) \times [0, 1)).$$ Hence $W(Q_V, \eta)$ is an orbifold with boundary where the boundary is the disjoint union of lens spaces $\{L(\Delta^{n-1}_j, \xi^j) \colon j = 1, \ldots, k\}$. Clearly orientations of $T^{n}$ and $Q$ induce an orientation of $W(Q, \lambda)$.
\end{proof}

\begin{remark}
If a vertex $V_j=F_{j_1} \cap \cdots \cap F_{j_n}$ of $Q$ such that the set of vectors $\{\eta(F_{j_1}), \ldots,$
$\eta(F_{j_{n}})\}$ form a basis of $\ZZ^{n}$, then $(T^n \times \Delta^{n-1}_{j})/\sim$ is
$\delta$-equivariantly homeomorphic to $S^{2n-1}$ with the standard $T^n$ action  for some automorphism $\delta$ of $T^n$.
\end{remark}

\section{Torus cobordism of $L(p; q_1, \ldots, q_n)$}\label{cobort2}
In this section we discuss the $T^{n+1}$-equivariant cobordism of lens spaces $L(p; q_1, \ldots, q_n)$.
First, we recall the definition of $T^{k}$-equivariant cobordism for $T^k$-manifolds where $k$ is a positive
integer.
\begin{defn}
Two same dimensional oriented closed smooth manifolds $M_1$ and $M_2$ with effective $T^k$-actions are said to be $T^k$-{\em equivariantly cobordant} if there exists an oriented $T^k$-manifold $W$ with boundary $\partial W $ such that $\partial W $ is equivariantly homeeomorphic to $ M_1 \sqcup (-M_2)$ under an
orientation preserving homeomorphism. Here $-M_2$ represents the reverse orientation of $M_2$. When a
$T^k$-manifold $M$ is the boundary of an oriented $T^k$-manifold with boundary, $M$ is called
$T^k$-{\em equivariantly oriented boundary}.
\end{defn}
From the definition of the manifold $L(\Delta^n, \xi)$ it is clear that $T^{n+1}$-action depends
on the characteristic function $\xi$. We denote the equivariant cobordism class of $L(p; q_1, \ldots, q_n)$
by $[L(p; q_1, \ldots, q_n)]_{\delta}$ where $\delta$ represents the action.
We discuss the equivariant cobordism of  $3$-dimensional, $5$-dimensional, and higher dimensional lens spaces
separately in the following subsections.
Given $a_0, \ldots, a_k \in \ZZ$, we denote greatest common divisor of them by
$\gcd\{a_0, \ldots, a_k\}$.

\subsection{Cobordism of $L(p, q)$, that is when $n=1$}\label{lpq}

\mbox{}\\

In this case $\Delta^1$ is an $1$-dimensional simplex, that is the closed interval $[0, 1]$.
Define $$\xi: \{\{0\}, \{1\}\} \to \ZZ^2$$ by $\xi(\{0\})= (1, 0)$ and $\xi(\{1\})= (-q, p)$ where
$\gcd\{q, p\} =1$ with $0\leq  |q| < p$. From Section \ref{quotient space} (also from Section $2$ of \cite{OR}, we get that the space $L(\Delta^1, \xi)$ is the lens space $L(p, q)$ with the natural $T^2$-action
coming from the group operation on the first factor of $T^2 \times \Delta^1$.

\begin{lemma}\label{lem1}
Let $(a,b), (c,d) \in \ZZ^2$ such that $|\det\{(a,b), (c, d)\}| = r > 1$ and $\gcd\{a, b\} = 1
= \gcd\{c,d\}$ . Then there exists $(e,f) \in \ZZ^2$ such that $|\det\{(a,b), (e,f)\}| =1$,
 $\gcd\{e,f\}$ $=1$ and $|\det\{(c, d), (e,f)\}| < r$.
\end{lemma}
\begin{proof}
First we prove when $(a, b)= (1, 0)$. Then $ |\det[(1,0), (c,d)]|$ $= r= |d|$.
Since $c, d $ are relatively prime and $r >1$, $d \neq 0,\pm 1$ and either
$c>d$ or $d >c$. Let $c>0, d >0 $ and $c>d$. Then $r$ is the area of the parallelogram $P$ in $\RR^2$ with
vertices $V_1=(0,0)$, $V_2=(1,0)$, $V_3=(c+1, d)$ and $V_4=(c, d)$. Clearly the length of $\{y=1\} \cap P$
is $1$. So $\{y=1\} \cap P \cap \ZZ^2$ is nonempty. It may contains only one point $ (u, 1)$, since
$c, d$ are relatively primes and the intersection $\{y=1\} \cap \{cy=dx\} \cap \ZZ^2$ is empty. From the
elementary geometry we get that the area $$|\det\{(c,d), (u, 1)\}| = |c-du|$$ of the parallelogram $P_1$
with vertices $(0,0)$, $(u, 1)$, $(c+u, d+1)$ and $(c,d)$ is less than $r$. Also $|\det\{(1,0), (u,1)\}|=1$
and $u, 1$ are relatively primes. For other possible values of $c$ and $ d$ we can prove similarly.

We now prove the case when $a, b$ are arbitrary relatively prime integers.
Since $a, b$ are relatively prime, there exists $(x, y) \in \ZZ^2$ such that $\det\{(a,b), (x,y)\}=1$. So there
exists $A \in SL(2, \ZZ)$ with $A (a,b) =(1, 0)$ and $A(x, y)=(0,1)$. Note that if $\gcd\{c,d\}=1$ then
$\gcd\{(a_{11}c+a_{12}d), (a_{21}c+a_{22}d)\}=1$ for any $(a_{ij}) \in SL(2, \ZZ)$. Using this and the
previous arguments we can prove the lemma for general case.
\end{proof}

\begin{lemma}\label{lem2}
Let $p, q$ be two relatively prime integers and $0 < q < p$. Then there exists a sequence of pairs $(q_1, p_1),
\ldots, (q_k, p_k)$ where $|\det\{(q_i,p_i), (q_{i+1}, p_{i+1})\}|=1$ for all $i=1, \ldots, k-1$ and
$(q_1, p_1)=(1,0), (q_k, p_k)=(q, p)$.
\end{lemma}
\begin{proof}
 The proof is the successive application of Lemma \ref{lem1}.
\end{proof}

\begin{theorem}\label{3lens}
Any lens space $L(p; q)$ is $T^2$-equivariantly oriented boundary.
\end{theorem}

\begin{proof}
Without loss of generality, we may assume that $0<q<p$. So by Lemma \ref{lem2}, there exists $(q_1, p_1), \ldots, (q_k, p_k) \in \ZZ^2$ where
$|\det\{(q_i,p_i), (q_{i+1}, p_{i+1})\}|=1$ for all $i=1, \ldots, k-1$ and $(q_1, p_1)=(1,0),
(q_k, p_k)=(-q, p)$. Consider the $(k+1)$-gon $Q^2$ with vertices $V_1, \ldots, V_{k+1}$. So the
edges are $V_1V_2, \dots, V_kV_{k+1}, V_{k+1}V_1$. 
Define a function
$$\eta: \{V_iV_{i+1} \mid i = 1, \ldots, k\} \to \ZZ^2$$
 by $\eta(V_iV_{i+1})= (q_i, p_i)$ for $i=1, \ldots, k$, see Figure \ref{len905} $(a)$.
Let $T_F$ be the circle subgroup of $T^2$ determined by $\eta(V_iV_{i+1})$ if $F=V_iV_{i+1}$ and
$T_F=T^2$ if $F =\{V_i\}$ for $i =2, \ldots, k$. Fix $T_F = 1$ if $F= V_1V_{k+1}$ and $F= Q^2$. We define an equivalence relation $\sim_b$ on the
product $T^2 \times Q^2$ by
\begin{equation}
 (t, x) \sim_b (s, y) ~ \mbox{if} ~ x = y ~\mbox{and} ~ts^{-1}\in T_F
\end{equation}
where $F$ is the unique face containing $x$ in its relative interior. We denote the quotient
space $(T^2 \times Q^2)/\sim_b$ by $W(Q^2, \eta)$. Following the construction 1.5 of \cite{DJ} and Section \ref{lensbdd}, one can show that the space $W(Q^2, \eta)$ is an oriented $T^2$-manifold with boundary and the boundary is $(T^2 \times V_1V_{k+1})/\sim_b$ which is equivariantly homeomorphic to $L(p; q)$.
\end{proof}

\begin{figure}
\centerline{
\scalebox{0.80}{
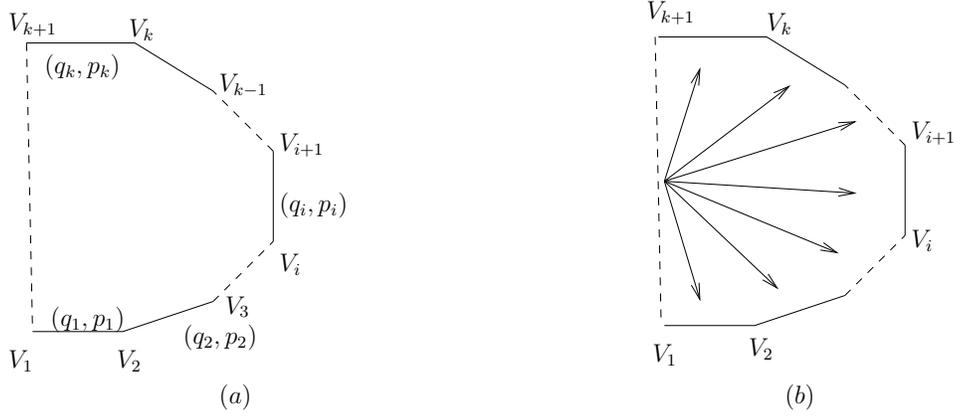
 }
 }
\caption {The map $\eta$ on $Q^2$ and a retraction of $Q^2$ respectively.}
\label{len905}
\end{figure}

\begin{corollary}
The lens space $L(p, q)$ bounds a simply connected 4-manifold. 
\end{corollary}
\begin{proof}
We adhere to the notations of the proof of Theorem \ref{3lens}.
Let $E= V_1V_2 \cup \cdots \cup V_{k}V_{k+1}$ and $\pi : W(Q^2, \eta) \to Q^2$ be the orbit map. Then $\pi^{-1}(E)$ is a deformation retract of $W(Q^2, \eta)$, a corresponding retraction on $Q^2$ is shown in Figure \ref{len905} $(b)$. Note that $\pi^{-1}(V_iV_{i+1}) \cong S^2$ for $i=1, k$ and $$\pi^{-1}(V_iV_{i+1}) \cap \pi^{-1}(V_{i+1}V_{i+2}) = \pi^{-1}(V_{i+1}) \cong \ast$$ for $i=1, \ldots, k-1$. Also for $i < j$, $\pi^{-1}(V_iV_{i+1}) \cap \pi^{-1}(V_{i+1}V_{i+2}) = \emptyset$. So $\pi^{-1}(E)$ is a wedge of $S^2$. Hence the corollary follows.
\end{proof}

\subsection{Cobordism of $L(p; q_1, q_2)$, that is when $n=2$}\label{lpqq}

\mbox{}\\

Consider a $2$-simplex $\Delta^2$ with vertices $v_0, v_1, v_2$ and edges $f_0=v_1v_2,
f_1=v_0v_2$ and $f_2=v_0v_1$. Let $p, q_1$ and $q_2$ be integers with $0 < q_1, q_2 < p$
such that $\gcd\{p, q_i\} =1$ for $i=1, 2$. Then the function $\xi: \mathcal{F}(\Delta^2)
\to \ZZ^3$ defined by $$\xi(f_0)= (-q_1, -q_2, p), \quad \xi(f_1)=(1,0,0) \quad \mbox{and} \quad \xi(f_2)= (0,1,0)$$
is a hyper characteristic function on $\Delta^2$ (see Figure \ref{len901} $(b)$) and
the manifold $L(\Delta^2, \xi) $ is the lens space $ L(p; q_1, q_2)$, see Section~\ref{quotient space}.

Consider a $3$-simplex $\Delta^3$ with vertices $\{V_0, V_1, V_2, V_3\}$ and facets
$\{F_i=V_0\ldots \widehat{V_i} \ldots V_3\}$ opposite to the vertex $V_i$ for $i =0, 1, 2, 3$.
We want to extend the hyper characteristic function $\xi$ of $\Delta^2$ to a rational
characteristic function $\eta: \mathcal{F}(\Delta^3) \to \ZZ^3$ on $\Delta^3$ (see Figure \ref{len903} (1)) satisfying $$\eta(F_0)=\xi(f_0) =(-q_1, -q_2, p), \quad \eta(F_1)=\xi(f_1)=(1,0,0),$$
$$ \eta(F_2)=\xi(f_2)=(0, 1, 0) \quad \mbox{and}~ \quad \eta(F_3) = (a, b, c)$$ such that $$\partial{W(\Delta^3_V, \eta)} = \bigsqcup_{i=0}^3 (T^3 \times \Delta^2_{i})/\sim_b$$ where 
$$(T^3 \times \Delta^2_{3})/\sim_b ~~~~~\cong L(p; q_1, q_2) \quad \mbox{and} \quad (T^3 \times \Delta^2_{i})/\sim_b ~~~\cong L(p_i; q_{i_1}, q_{i_2})$$ with $p_i < p$ for $i=0, 1, 2$.

We then apply the similar procedure to each lens spaces $L(p_i; q_{i_1}, q_{i_2})$, i.e., to the induced hyper characteristic function on $\Delta^2_{i}$, until we get a $T^3$-equivariant cobordism of $L(p; q_1, q_2)$ to the sum of $L(1; 1, 1) \cong S^5$ which is $T^3$-equivariantly a boundary.

The above procedure depends on the choice of the vector $\eta(F_3) = (a,b,c) \in \ZZ^3$. Next we formulate the existence of such vector in the following question. Let $$\mathcal{L}(3)= \{(-q_1, -q_2, p) \in \ZZ^3 \mid \gcd\{p, q_1\} = \gcd\{p, q_2\}=1 ~ \mbox{and} ~ 0 \leq  q_1, q_2 < p\}.$$

\begin{ques}\label{ques1}
For a given $\eta_0 = (-q_1, -q_2, p) \in \mathcal{L}(3)$, and $\eta_1 = (1, 0, 0)$ and $\eta_2 = (0, 1, 0)$
does there exist $\eta_3 = (-a, -b, c) \in \ZZ^3$ such that $0\leq a, b \leq c$, $1 \leq c$ satisfying the following: 
\begin{enumerate}
\item $\{\eta_i, \eta_j\}$ forms a part of a basis of $\ZZ^3$ for any distinct $i,j \in \{0, 1, 2,3\}$,

\item $0 <|\det \eta_{ijk}| < p$ for any distinct  $i, j, k \in \{0, 1, 2, 3\}$ and $\{i, j, k\} \neq \{0, 1, 2\}$ where $\eta_{ijk}$ is the $3 \times 3$ matrix with $\eta_i, \eta_j, \eta_k$ as its row vectors.
\end{enumerate}
\end{ques}

Note that the condition (1) in the above Question \ref{ques1} is equivalent to 
\begin{equation}\label{rach1}
 \gcd\{a, c\}=1, ~ \gcd\{b, c\}=1~\mbox{and} ~ \gcd\{bp- cq_2, ap- cq_1, bq_1-aq_2\}=1.
\end{equation}
In Question \ref{ques1}, the condition (2) is equivalent to 
\begin{equation}
c < p, \quad |ap-cq_1| < p \quad \mbox{and} \quad |bp-cq_2| < p.
\end{equation}

So we can restate Question \ref{ques1} into the following purely number theoretical question:

\begin{ques}
For a given vector $(-q_1, -q_2, p) \in \mathcal{L}(3)$, do there exist integers $a, b, c \in \ZZ$ such that
\begin{enumerate}
\item $\gcd\{a, c\}=1, ~ \gcd\{b, c\}=1~\mbox{and} ~ \gcd\{bp- cq_2, ap- cq_1, bq_1-aq_2\}=1,$

\item $c < p, \quad |ap-cq_1| < p \quad \mbox{and} \quad |bp-cq_2| < p.$
\end{enumerate}
\end{ques}

\begin{obsr}\label{obsr1}
Suppose $\eta_i, \eta_j, \eta_k$ are vectors as in condition (2) of Question \ref{ques1}. Then we can always find $A \in GL(3, \ZZ)$ such that $A(\eta_i) = (-q_1^{\prime}, -q_2^{\prime}, p^{\prime})$ with $0 < p^{\prime} < p$, $A(\eta_j) = (1, 0, 0)$ and $A(\eta_k) = (0, 1, 0)$. 

Indeed, by condition (1) of Question \ref{ques1}, $\eta_j, \eta_k$ is a part of a basis of $\ZZ^3$. Extend it to a basis $\{\eta_j, \eta_k, \eta_{\ell}\}$. So $\eta_i = r_j \eta_j + r_k \eta_k + r_{\ell} \eta_{\ell}$ for unique $r_j, r_k, r_{\ell} \in \ZZ$. Then $r_{\ell} \neq 0$, since $0 < |\det \eta_{ijk}|$. Then define $A : \ZZ^3 \to \ZZ^3$ to be the linear map determined by $$A(\eta_j) =(1, 0, 0), ~~ A(\eta_k) = (0, 1, 0), ~~ A(\eta_{\ell}) = \epsilon (0, 0, 1)$$ where $\epsilon = 1 $ if $ r_{\ell} > 0$ and $\epsilon = -1 $ if $ r_{\ell} < 0$.  Then $A \in GL(3, \ZZ)$ and $|\det A| =1$, and if we let $(-q_1^{\prime}, -q_2^{\prime}, p^{\prime}) = (r_j, r_k, \epsilon r_{\ell})$ then $A(\eta_i) = (-q_1^{\prime}, -q_2^{\prime}, p^{\prime})$ and $p^{\prime} > 0$. That is, $A$ satisfies the desired properties. 

Now by condition (2) of Question \ref{ques1},
$$p^{\prime} = |\det B| = |\det A^{-1} \eta_{ijk}| = |\det A^{-1}| |\eta_{ijk}| = |\det \eta_{ijk}| < p$$
where $B$ is the $3 \times 3$ matrix with $(1, 0, 0)$, $(0, 1, 0)$, $(-q_1^{\prime}, -q_2^{\prime}, p^{\prime})$ as its row vectors. 

Moreover, since $A \in GL(3, \ZZ)$, the condition (1) of Question \ref{ques1} implies that any two row vectors of $B$ form a part of a basis of $\ZZ^3$, which means that
$$\gcd \{p^{\prime}, q_1^{\prime}\} = \gcd \{p^{\prime}, q_2^{\prime}\} = 1.$$ 

Since the lens space $L(p^{\prime}; q_1^{\prime}, q_2^{\prime})$ is determined by the number $p^{\prime} \in \NN$ and $\mod p^{\prime}$ classes $[q_1^{\prime}]$ and $[q_2^{\prime}]$, one can replace $q_i^{\prime}$ by $q_i^{\prime} + k p^{\prime}$, if necessary, so that we may assume $0 \leq q_1^{\prime}, q_2^{\prime} < p^{\prime}$. This implies that $(-q_1^{\prime}, -q_2^{\prime}, p^{\prime}) \in \mathcal{L}(3)$.
\end{obsr}

\begin{remark}
At this moment, the authors do not know whether Question \ref{ques1} has the positive answer. If there is an algorithm to find such $(-a,-b,c)$ as in Question \ref{ques1} for any $(-q_1, -q_2, p) \in \mathcal{L}(3)$, then we can explicitly construct an oriented 6-dimensional $T^3$-manifold with boundary $L(p; q_1, q_2)$ as we will see in Theorem \ref{eqcob3}.
\end{remark}

 Let 
\begin{align*} 
 \mathfrak{N}:= max\big\{& p \in \NN ~~\big|~~ \mbox{for any} ~ (-q_1^{\prime}, -q_2^{\prime}, p^{\prime}) \in \mathcal{L}(3) ~~\mbox{with} ~~ p^{\prime} \leq p ~\mbox{Question}~ \ref{ques1}~\\
  & \mbox{has the positive answer} \big\}.
 \end{align*}
Note that Question \ref{ques1} has the positive answer for arbitrary $(-q_1, -q_2, p) \in \mathcal{L}(3)$ if and only if $\mathfrak{N} = \infty$. It is checked that Question \ref{ques1} has the positive answer for $(-q_1, -q_2, p) \in \mathcal{L}(3)$ with $0 < p \leq 50$, i.e., $\mathfrak{N} \geq 50$.

\begin{theorem}\label{eqcob3}
The lens space $L(p; q_1, q_2)$ is $T^3$-equivariantly the boundary of an oriented manifold if $(-q_1, -q_2, p) \in \mathcal{L}(3)$ with $p \leq \mathfrak{N}$.
\end{theorem}

\begin{proof}
Let $(-q_1, -q_2, p) \in \mathcal{L}(3)$ with $p \leq \mathfrak{N}$ and $\Delta^3$ be a $3$-simplex with the set of vertices $\{V_0, V_1, V_2, V_3\}$. Let $F_i$ be the facet of
$\Delta^3$ opposite to the vertex $V_i$ for $i=0,1,2,3$. From the definition of $\mathfrak{N}$, one can
conclude that there exists $(-a,-b,c) \in \mathcal{L}(3)$ which satisfy all conditions of Question \ref{ques1}. Define the function $\eta: \mathcal{F}(\Delta^3) \to \ZZ^3$ by $$\eta(F_0) =(-q_1, -q_2, p), ~~
\eta(F_1)=(1,0,0), ~~ \eta(F_2)=(0, 1, 0) ~~\mbox{and}~~ \eta(F_3) = (-a, -b, c).$$ Then $\eta$ is a rational
characteristic function on $\Delta^3$. By Lemma \ref{lpbdd} we get an oriented $T^3$-manifold  
$W(\Delta^3_V, \eta)$ with boundary. The boundaries of $W(\Delta_V^3, \eta)$ are the lens spaces
$(T^3 \times \Delta^2_{i})/\sim_b$ for $i=0,1,2,3$ where $\sim_b$ is the equivalence relation
in (\ref{equieta}). The simple polytope $\Delta_V^3$ and the rational characteristic
function $\eta$ are given in the Figure \ref{len903}.

\begin{figure}[ht]
\centerline{
\scalebox{0.74}{
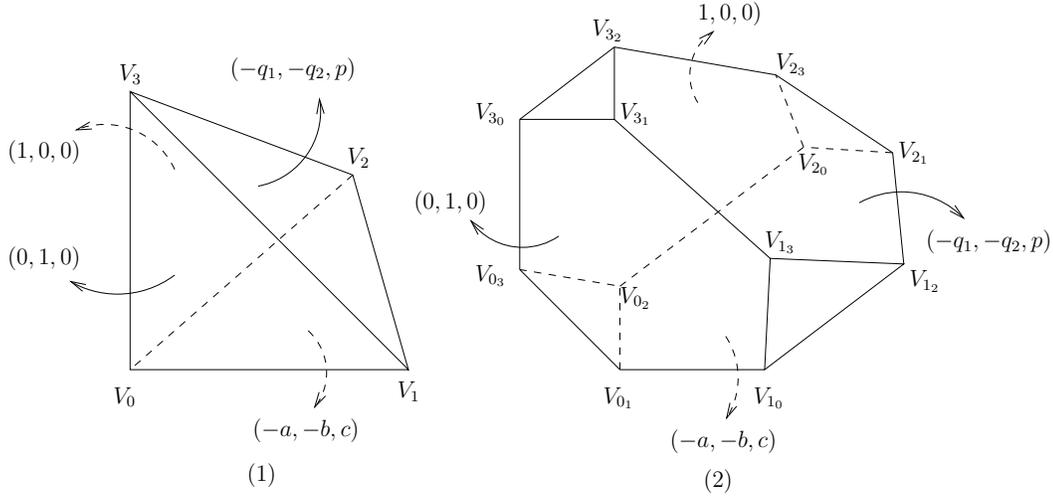
 }
 }
\caption {$\Delta^3$ and vertex cut $\Delta_V^3$ of $\Delta^3$.}
\label{len903}
\end{figure}
Let $\Delta^2_{i}$ be the facet of $\Delta_V^3$ obtained from the
vertex cut of $\Delta^3$ at the vertex $V_i$. Then the vertices of $\Delta^2_i$ are given by $V_{i_j}=\Delta^2_{i}\cap F_k \cap F_l$ with $\{i, j, k,l\} = \{0,1,2,3\}$. Define $$\xi^i: \mathcal{F}(\Delta^2_{i}) \to \ZZ^3 \quad \mbox{by} \quad \xi^i(V_{i_j}V_{i_k})= \eta(F_l)$$ where $\{i,j,k,l\} =\{0, 1,2,3\}$. Then for each $i \in \{0, 1,2,3\}$, $\xi^i$ is a hyper characteristic function on $\Delta^2_{i}$.\\

When $i=0$, then $$\rm{Im}(\xi^0) =\{\eta(F_1), \eta(F_2), \eta(F_3)\}= \{(1, 0, 0), (0, 1, 0), (-a, -b, c)\}.$$ Hence, by Section \ref{quotient space}, the boundary component $L(\Delta^2_{0}, \xi^0) =(T^3 \times \Delta^2_{0})/\sim_b$ of $W(\Delta^3_V, \eta)$ is the lens space $L(c; a, b)$ with the natural action $\delta_0$ of $T^3$. Note that $c < p$.\\

When $i=1$, then $\rm{Im}(\xi^1)=\{\eta(F_0), \eta(F_2), \eta(F_3)\}$ and by condition (2) of Question \ref{ques1} $$0 <|\det[\eta(F_0), \eta(F_2),
\eta(F_3)]|= |a p + cq_1| < p.$$ The Observation \ref{obsr1} implies that there is an automorphism $A_1$ of $\ZZ^3$ over $\ZZ$, which induces an automorphism $\delta_1$ of $T^3$, such that 
\begin{equation}\label{hypch1}
A_1(\eta(F_0)) = (-q_{1_1}^{\prime}, -q_{1_2}^{\prime}, p_1), \quad A_1(\eta(F_2)) = (1,0, 0) \quad \mbox{and} \quad A_1(\eta(F_3)) = (0, 1,0)
\end{equation}
 where $p_1 > 0$. Since $|\det(A_1)| =1$, $p_1 =|a p + cq_1| < p$. Let
$q_{1_1}^{\prime} \equiv q_{1_1} \mod (p_1)$ and $ q_{1_2}^{\prime} \equiv q_{1_2} \mod (p_1)$ where $0 \leq q_{1_1}, q_{1_2} \leq p_1$. The composition $A_1 \circ \xi^1$ is a hyper characteristic function on $\Delta^2_{1}$. Then from \eqref{hypch1} and definition of hyper characteristic function, one can easily show that $$\gcd\{p_1, q_{1_1}\}= 1 = \gcd\{p_1, q_{1_2}\}.$$ Hence the boundary component $L(\Delta^2_{1}, \xi^1) = (T^3 \times \Delta^2_{1})/\sim_b$ of $W(\Delta^3_V, \eta)$ is $\delta_1$-equivariantly homeomorphic to the lens space $L(p_1; q_{1_1}, q_{1_2})$, where $p_1< p < \mathfrak{N}$.\\

When $i=2$, $\rm{Im}(\xi^2)=\{\eta(F_0), \eta(F_1), \eta(F_3)\}$. By condition (2) of Question \ref{ques1}
 $$0 < |\det[\eta(F_0), \eta(F_1), \eta(F_3)]|= |bp - cq_2| < p.$$ The Observation \ref{obsr1} implies that there is an automorphism $A_2$ of $\ZZ^3$ over $\ZZ$, which induces an automorphism $\delta_2$ of $T^3$, such that 
\begin{equation}\label{hypch2}
A_2(\eta(F_0)) = (-q_{2_1}^{\prime}, -q_{2_2}^{\prime}, p_2), \quad A_2(\eta(F_1)) = (1,0, 0) \quad \mbox{and} \quad A_2(\eta(F_3)) = (0, 1,0)
\end{equation}
 where $p_2 > 0$. Since $|\det(A_2)| =1$, $p_2 =|b p + cq_2| < p$. Let
$q_{2_1}^{\prime} \equiv q_{2_1} \mod (p_2)$ and $ q_{2_2}^{\prime} \equiv q_{2_2} \mod (p_2)$ where $0 \leq q_{2_1}, q_{2_2} \leq p_2$. The composition $A_2 \circ \xi^2$ is a hyper characteristic function on $\Delta^2_{2}$. Then from \eqref{hypch2} and definition of hyper characteristic function, one can easily show that $$\gcd\{p_2, q_{2_1}\}= 1 = \gcd\{p_2, q_{2_2}\}.$$ Hence the boundary component $L(\Delta^2_{2}, \xi^2) = (T^3 \times \Delta^2_{2})/\sim_b$
of $W(\Delta^3_V, \eta)$ is $\delta_2$-equivariantly homeomorphic to the
lens space $L(p_2; q_{2_1}, q_{2_2})$, where $p_2 < p < \mathfrak{N}$.\\


When $i=3$, 
$${\rm{Im}}(\xi^3) = \{\eta(F_0), \eta(F_1), \eta(F_2)\} = \{(1, 0, 0), (0,1,0), (-q_1, -q_2, p)\}.$$ Hence, by Section \ref{quotient space}, the boundary component $L(\Delta^2_{3}, \xi^3) =(T^3 \times \Delta^2_{3})/\sim_b$ of $W(\Delta^3_V, \eta)$ is the lens space $L(p, q_{1}, q_{2})$ with the natural action $\delta_3$ of $T^3$.\\

From the above discussion we get that $L(p; q_1, q_2)$ is $T^3$-equivariant cobordant to $$L(c; a, b) \sqcup L(p_1; q_{1_1}, q_{1_2}) \sqcup L(p_2; q_{2_1}, q_{2_2})$$ where $0 < c, p_1, p_2 < p$. That is, 
\begin{equation}\label{cobeq}
 [L(p; q_1, q_2)]_{\delta_3} = [L(c; a, b)]_{\delta_0} + [L(p_1; q_{1_1}, q_{1_2})]_{\delta_1} + [L(p_2; q_{2_1}, q_{2_2})]_{\delta_2}
\end{equation}
 where $0 < c, p_1, p_2 < p$ and $\delta_i$'s are the corresponding actions of $T^3$ on the respective lens spaces.

 Since $c, p_1, p_2 < \mathfrak{N}$, we may continue the previous constructive technique on the lens spaces of the right hand side of \eqref{cobeq}, to show that $[L(p; q_1, q_2)]_{\delta_3}$ is zero. Actually, we can construct an oriented $T^3$-manifold with boundary
where the boundary is the lens space $L(p; q_1,q_2)$ by gluing the successive corresponding boundaries via
orientation preserving equivariant maps.
\end{proof}

\begin{corollary}
The lens space $L(p; q_1, q_2)$ is $T^3$-equivariantly the boundary of an oriented manifold if  two
integers $q_1$ and  $q_2$ are relatively prime and $q_1, q_2 \leq \mathfrak{N}$.
\end{corollary}
\begin{proof}
We may assume $q_1, q_2 > 0$.
Since $q_1 $ and $ q_2$ are relatively prime, then $(-a, -b, c) = (0, 0, 1)$ satisfy all the conditions in Question \ref{ques1}. Hence arguing similarly as the proof of Theorem \ref{eqcob3}, one can show that
$$ [L(p; q_1, q_2)]_{\delta_3} = [L(1; 0, 0)]_{\delta_0} + [L(p_1; q_{1_1}, q_{1_2})]_{\delta_1} + [L(p_2; q_{2_1}, q_{2_2})]_{\delta_2}$$ where $p_1 = |ap+cq_1|=q_1$ and $p_2 = |bp+cq_2|=q_2$. Note that $L(1; 0,0) = S^5$. So $L(1;0,0)$ is $T^3$-equivariantly oriented boundary. Since $0 \leq q_1, q_2 \leq \mathfrak{N}$, Theorem \ref{eqcob3} implies that $L(p_1; q_{1_1}, q_{1_2})$ and $L(p_2; q_{2_1}, q_{2_2})$ are $T^3$-equivariantly oriented boundary. Therefore the corollary follows.
\end{proof}

\begin{corollary}
The lens space $L(p; q_1, q_2)$ is $T^3$-equivariantly the boundary of an oriented manifold if $q_2=kq_1$ and
$q_1, p-q_2 \leq \mathfrak{N}$.
\end{corollary}
\begin{proof} We may assume $q_1$ is positive. 
 Since $q_2= kq_1$ and $(-q_1, -q_2, p) \in \mathcal{L}(3)$, then $(-a, -b, c) = (0, -1, 1)$ satisfy all the conditions in Question \ref{ques1}. Hence arguing similarly as the proof of Theorem \ref{eqcob3}, one can show that
$$ [L(p; q_1, q_2)]_{\delta_3} = [L(1; 1, 0)]_{\delta_0} + [L(p_1; q_{1_1}, q_{1_2})]_{\delta_1} + [L(p_2; q_{2_1}, q_{2_2})]_{\delta_2}$$ where $p_1 = |ap+cq_1|=q_1$ and $p_2 = |bp+cq_2|=p -q_2$. Note that $L(1; 1,0) \cong S^5$. So $L(1;1,0)$ is $T^3$-equivariantly oriented boundary. Since $0 < q_1, p-q_2 \leq \mathfrak{N}$, Theorem \ref{eqcob3} implies that $L(p_1; q_{1_1}, q_{1_2})$ and $L(p_2; q_{2_1}, q_{2_2})$ are $T^3$-equivariantly oriented boundary. Therefore the corollary follows.
\end{proof}

\begin{corollary}
The lens space $L(p; q_1, q_2)$ is $T^3$-equivariantly the boundary of an oriented manifold if  two
integers $p-q_1$ and  $p-q_2$ are relatively prime and $p-q_1, p-q_2 \leq \mathfrak{N}$.
\end{corollary}
\begin{proof}
We may assume $q_1, q_2 > 0$.
Since $p-q_1 $ and $p- q_2$ are relatively prime, then $(-a, -b, c) = (-1, -1, 1)$ satisfy all the conditions in Question \ref{ques1}. Hence arguing similarly as the proof of Theorem \ref{eqcob3}, one can show that
$$ [L(p; q_1, q_2)]_{\delta_3} = [L(1; 0, 0)]_{\delta_0} + [L(p_1; q_{1_1}, q_{1_2})]_{\delta_1} + [L(p_2; q_{2_1}, q_{2_2})]_{\delta_2}$$ where $p_1 = |ap+cq_1|=p-q_1$ and $p_2 = |bp+cq_2|=p-q_2$. Note that $L(1; 1,1) \cong S^5$. So $L(1;1,1)$ is $T^3$-equivariantly oriented boundary. Since $0 \leq q_1, q_2 \leq \mathfrak{N}$, Theorem \ref{eqcob3} implies that $L(p_1; q_{1_1}, q_{1_2})$ and $L(p_2; q_{2_1}, q_{2_2})$ are $T^3$-equivariantly oriented boundary. Therefore the corollary follows.
\end{proof}

\subsection{Cobordism of $L(p; q_1, \ldots, q_n)$ when $n > 2$}\label{lpqqq}

\mbox{}\\

Let $\Delta^n$ be an $n$-simplex with vertices $\{v_0, v_1, \ldots, v_n\}$,  and let $F_i$ be the facet of $\Delta^n$ which does not contain the vertex $v_i$ for $i=0, \ldots, n$. Let $\{e_i \mid i=1, \ldots, n+1\}$ be the standard basis of $\ZZ^{n+1}$. Define a function $$\xi \colon \{F_i \mid i=0, \ldots, n\} \to \ZZ^{n+1}$$ by $\xi(F_0)= (-q_1, \ldots, -q_n, p)$ and $\xi(F_i)=e_i$ for $i=1, \ldots, n$ where $p$ is relatively prime to each $q_i$ for $i=1, \ldots, n$ with $0 \leq  q_1, \cdots, q_n < p$. So $\xi$ is a hyper characteristic function on $\Delta^n$. Then from Section \ref{quotient space},  we get $L(\Delta^n, \xi) = L(p; q_1, \ldots, q_n)$.

 Let $$\mathcal{L}(n)=\{(-q_1, \ldots, -q_n, p) \in \ZZ^{n+1} \mid  \gcd\{p, q_i\} = 1 ~~ \mbox{for} ~~
i=1, \ldots, n ~~ \mbox{with}~ 0\leq  q_1, \ldots, q_n \leq p\}.$$
Analogous version of Question \ref{ques1} for $n > 2$ case is also valid, and if the answer to this question is positive, we can construct an oriented $T^{n+1}$-manifold with boundary $L(p; q_1, \ldots, q_n)$. That is, we can give an alternative proof of the following: Any lens space $L(p; q_1, \ldots, q_n)$ is a $T^{n+1}$-equivariantly oriented boundary. Similarly to the $n=2$ case, we ask the following Question.

\begin{ques}\label{ques2}
For a given $\eta_0 = (-q_1, \dots, -q_n, p) \in \mathcal{L}(n)$, and $\eta_i = e_i$ for $i=1, \ldots, n$, does there exist $\eta_{n+1} = (-a_1, \ldots -a_n, b) \in \ZZ^{n+1}$ with $0 \leq a_1, \ldots, a_n \leq b$ and $1 \leq b$ such that 
\begin{enumerate}
\item $\{\eta_0, \ldots, \widehat{\eta_i}, \ldots, \widehat{\eta_j}, \ldots, \eta_{n+1}\}$ forms a part of a basis of $\ZZ^{n+1}$ for any distinct $i,j \in \{0, \ldots, n+1\}$ with $i < j$, where \ $\widehat{}$ \ represents the omission of the corresponding entry,\\

\item $0 <|\det \eta_{0 \ldots \widehat{i} \ldots n+1}| < p$ for any $\{0 \ldots \widehat{i} \ldots n+1\} \subset \{0, \ldots, n+1\}$ and $i \neq n+1$ where $\eta_{0 \ldots \widehat{i} \ldots n+1}$ is the $n+1 \times n+1$ matrix with $\eta_0, \ldots, \widehat{\eta_i}, \ldots, \eta_{n+1}$ as its row vectors.\\
\end{enumerate}
\end{ques}

At this moment the authors do not know whether Question \ref{ques2} has the positive answer. Similarly to  subsection \ref{lpqq} one can restate the Question \ref{ques2} in the form of number theoretical question.

\begin{ques}
For a given vector $(-q_1, \ldots, -q_n, p) \in \mathcal{L}(n)$, do there exist integers $a_1, \ldots, a_n, b \in \ZZ$ such that
\begin{enumerate}
\item 
\begin{align*}
&\gcd\{a_i, b\}=1, \textrm{for} ~ i= 1, \ldots, n, \textrm{and}\\
& \gcd\{a_2p-bq_2, a_1p-bq_1, a_1q_2 - a_2q_1\}=1,\\
& \gcd\{a_3p-bq_3, a_1p-bq_1, a_1q_3-a_3q_1\}=1,\\
& \gcd\{a_4p-bq_4, a_1p-bq_1, a_1q_3-a_3q_1\}=1,\\
& ...\\
& \gcd\{a_np-bq_n, a_{n-1}p-bq_{n-1}, a_{n-1}q_n-a_nq_{n-1}\}=1.
\end{align*}
\item
$b < p, \quad \mbox{and} \quad |a_ip-bq_i| < p \quad \mbox{for} \quad i=1, \ldots, n$?
\end{enumerate}
\end{ques}

Suppose $\{\eta_{i_1}, \ldots, \eta_{i_{n+1}}\} = \{\eta_0, \ldots, \widehat{\eta_i}, \ldots, \eta_{n+1}\}$ are the vectors as in condition (2) of Question \ref{ques2}. Then we can always find $A \in GL(n, \ZZ)$ such that 
\begin{align*}
& A(\eta_{i_j}) = e_{j} ~~\mbox{for} ~~ 1 \leq j \leq n, ~~ \mbox{and}\\
& A(\eta_{i_{n+1}}) = (-q_1^{\prime}, \ldots, -q_n^{\prime}, p^{\prime}) ~~\mbox{with} ~~ 0 < p^{\prime} < p.
\end{align*}
The proof of this claim is similar to Observation \ref{obsr1}.
Let
\begin{align*}
\mathfrak{M}=max\{& p \in \NN ~\mid ~\mbox{for any} ~ (-q_1^{\prime}, \ldots, -q_n^{\prime}, p^{\prime}) \in \mathcal{L}(n) ~~\mbox{with} ~~ p^{\prime} \leq p ~~\mbox{the Question}~~ \ref{ques2}~\\
& \mbox{has a solution}\}.
\end{align*}

\begin{theorem}\label{eqcob4}
The lens space $L(p; q_1, \ldots, q_n)$ is $T^{n+1}$-equivariantly the boundary of an oriented manifold if
$(q_1, \ldots, q_n, p) \in \mathcal{L}(n)$ with $p \leq \mathfrak{M}$.
\end{theorem}
\begin{proof}
The proof is similar to the proof of Theorem \ref{eqcob3}. So we briefly outline the proof of this theorem.
Let $\Delta^{n+1}$ be the $(n+1)$-dimensional simplex with vertices $V_0, V_1, ..., V_{n+1}$
and facets $F_0, \ldots, F_{n+1}$ where $F_i$ does not contain the vertex $V_i$ for $i=0, \ldots, n+1$.
Since $p \leq \mathfrak{M}$, from the definition of $\mathfrak{M}$, one can
conclude that there exists $(-a_1, \ldots, -a_n,b) \in \mathcal{L}(n)$ which satisfy all conditions of Question \ref{ques2}. Define a function $$\eta: \{F_0, \ldots, F_{n+1}\} \to \ZZ^{n+1}$$ by
 $$\eta(F_{0})=(-q_1, \ldots, -q_n, p), \eta(F_{n+1})= (-a_1, \ldots, -a_n, b)
~\mbox{and}~ \eta(F_i)=e_i ~\mbox{for}~ i=1, \ldots, n.$$ Since $p$ is relatively prime to each $q_i$ and $b$ is relatively prime to each $a_i$ for $i=1, \ldots, n$, the function $\eta$ is an rational characteristic function on $\Delta^{n+1}$. Hence by Lemma \ref{lpbdd}, $W(\Delta^{n+1}_V, \eta)$ is an oriented $T^{n+1}$-manifold with boundary where the boundaries are the generalized lens spaces $(T^{n+1} \times \Delta^n_{i})/\sim_b$ for $i=0, 1, \ldots, n+1$.

Note that the facets of $\Delta^n_{i}$ are $\{\Delta^n_{i} \cap F_j: j \neq i ~ \mbox{and}
~ j \in {0, \ldots, n+1}\}$. The restriction of $\eta$ on the facets of $\Delta^n_{i}$ is given by
$$\xi^i(\Delta^n_{i} \cap F_j) = \eta(F_j) ~ j \neq i ~ \mbox{and}~ j \in {0, \ldots, n+1}.$$
So $\xi^i$ is a hyper characteristic function on $\Delta^n_{i}$ and $(T^{n+1} \times \Delta^n_{i})/\sim_b
= L(\Delta^n_{i}, \xi^i)$. Similarly as in the proof of Theorem \ref{eqcob3}, we can show that
$$L(\Delta^n_{{n+1}}, \xi^{n+1}) \cong L(p; q_1, \ldots, q_n), \quad L(\Delta^n_{0}, \xi^{0})
\cong L(b; a_1, \ldots, a_n)$$ and $L(\Delta^n_{i}, \xi^i)$ is $\delta_i$-equivariantly homeomorphic to
the lens spaces $ L( p_i; q_{i_1}, \ldots, q_{i_n})$ for $i =1, \ldots, n$ where $0 \leq q_{i_1}, \cdots, 
q_{i_n} < p_i < p \leq \mathfrak{M}$. That is,
\begin{equation}\label{cobeq2}
 [L(p; q_1, \ldots, q_n)]_{\delta_n} = [L(b; a_1, \ldots, a_n)]_{\delta_0} + [L(p_1; q_{1_1}, \ldots, q_{1_n})]_{\delta_1} + \cdots + [L(p_n; q_{n_1}, \ldots, q_{n_n})]_{\delta_n}
 \end{equation}
 where $\delta_i$'s are the corresponding torus action. 

Since $b, p_1, \ldots, p_n < \mathfrak{M}$, continuing this process on the right hand side of \eqref{cobeq2}, we can show $L(p; q_1, \ldots, q_n)$ is $T^{n+1}$-equivariantly oriented boundary of an oriented manifold. Moreover, we can construct an oriented $T^{n+1}$-manifold with boundary where the boundary is the lens space $L(p; q_1, \ldots, q_n)$ by gluing the successive corresponding boundaries via orientation preserving equivariant maps.
\end{proof}

\begin{corollary}\label{main corollary}
The lens space $L(p; q_1, \ldots, q_n)$ is equivariantly the boundary of an oriented manifold if any two
integers of  $\{q_1, \ldots, q_n\} $ are relatively prime and $q_1, \ldots, q_n \leq \mathfrak{M}$.
\end{corollary}
\begin{proof}
Since any two integers of the set $\{q_1, \ldots, q_n, p\}$ are relatively prime, then the vector $(-a_1, \ldots, -a_n, b) = (0, \ldots, 0, 1)$ satisfies all conditions in Question \ref{ques2}. We consider $\eta(F_{n+1})= (0, \ldots, 0, 1)$. So form the proof of Theorem \ref{eqcob4}, we get  
$$ [L(p; q_1, \ldots, q_n)]_{\delta_n} = [L(1; 0, \ldots, 0)]_{\delta_0} + [L(p_1; q_{1_1}, \ldots, q_{1_n})]_{\delta_1} + \cdots + [L(p_n; q_{n_1}, \ldots, q_{n_n})]_{\delta_n}.$$
where $p_i = |a_ip-bq_i|=q_i < p < \mathfrak{M}$. Note that $L(1; 0, \ldots, 0) = S^{2n+1}$ and so $L(1;0, \ldots, 0)$ is $T^{n+1}$-equivariantly oriented boundary. Since $0 \leq q_1, \ldots, q_n < \mathfrak{M}$, Theorem \ref{eqcob4} implies that $L(p_i; q_{i_1}, \ldots, q_{i_n})$ is $T^{n+1}$-equivariantly oriented boundary. Therefore the corollary follows.
\end{proof}

\begin{remark}
The oriented non-equivariant cobordism class of $L(p; q_1, \ldots, q_n)$ is zero, since all the Stiefel-Whitney numbers of it are zero.
\end{remark}

{\bf Acknowledgement.} 
The authors thanks Richard K. Guy for helpful discussion regarding Question \ref{ques1}.
 The first author thanks Korea Advanced Institute of Science and Technology, Pacific Institute for Mathematical Sciences, University of Regina, and University of Calgary for financial support.

\renewcommand{\refname}{References}

\bibliographystyle{alpha}
\bibliography{bibliography.bib}

\vspace{1cm}

\vfill

\end{document}

%% file: len901.pdf_t
\begin{picture}(0,0)%
\includegraphics{len901.pdf}%
\end{picture}%
\setlength{\unitlength}{4144sp}%
\begingroup\makeatletter\ifx\SetFigFont\undefined%
\gdef\SetFigFont#1#2#3#4#5{%
  \reset@font\fontsize{#1}{#2pt}%
  \fontfamily{#3}\fontseries{#4}\fontshape{#5}%
  \selectfont}%
\fi\endgroup%
\begin{picture}(5250,1893)(2011,-2605)
\put(2656,-2176){\makebox(0,0)[lb]{\smash{{\SetFigFont{12}{14.4}{\rmdefault}{\mddefault}{\updefault}{\color[rgb]{0,0,0}$(4,1,0)$}%
}}}}
\put(3511,-1411){\makebox(0,0)[lb]{\smash{{\SetFigFont{12}{14.4}{\rmdefault}{\mddefault}{\updefault}{\color[rgb]{0,0,0}$(3,2,4)$}%
}}}}
\put(2026,-1411){\makebox(0,0)[lb]{\smash{{\SetFigFont{12}{14.4}{\rmdefault}{\mddefault}{\updefault}{\color[rgb]{0,0,0}$(0,2,3)$}%
}}}}
\put(6166,-2131){\makebox(0,0)[lb]{\smash{{\SetFigFont{12}{14.4}{\rmdefault}{\mddefault}{\updefault}{\color[rgb]{0,0,0}$(1, 0, 0)$}%
}}}}
\put(2881,-2536){\makebox(0,0)[lb]{\smash{{\SetFigFont{12}{14.4}{\rmdefault}{\mddefault}{\updefault}{\color[rgb]{0,0,0}$(a)$}%
}}}}
\put(6301,-2536){\makebox(0,0)[lb]{\smash{{\SetFigFont{12}{14.4}{\rmdefault}{\mddefault}{\updefault}{\color[rgb]{0,0,0}$(b)$}%
}}}}
\put(5041,-1501){\makebox(0,0)[lb]{\smash{{\SetFigFont{12}{14.4}{\rmdefault}{\mddefault}{\updefault}{\color[rgb]{0,0,0}$(0,  1,0)$}%
}}}}
\put(6526,-1411){\makebox(0,0)[lb]{\smash{{\SetFigFont{12}{14.4}{\rmdefault}{\mddefault}{\updefault}{\color[rgb]{0,0,0}$(-q_1,- q_2, p)$}%
}}}}
\put(5716,-2131){\makebox(0,0)[lb]{\smash{{\SetFigFont{12}{14.4}{\rmdefault}{\mddefault}{\updefault}{\color[rgb]{0,0,0}$V_0$}%
}}}}
\put(7246,-2041){\makebox(0,0)[lb]{\smash{{\SetFigFont{12}{14.4}{\rmdefault}{\mddefault}{\updefault}{\color[rgb]{0,0,0}$V_1$}%
}}}}
\put(5671,-916){\makebox(0,0)[lb]{\smash{{\SetFigFont{12}{14.4}{\rmdefault}{\mddefault}{\updefault}{\color[rgb]{0,0,0}$V_2$}%
}}}}
\put(2071,-2086){\makebox(0,0)[lb]{\smash{{\SetFigFont{12}{14.4}{\rmdefault}{\mddefault}{\updefault}{\color[rgb]{0,0,0}$V_0$}%
}}}}
\put(3736,-2086){\makebox(0,0)[lb]{\smash{{\SetFigFont{12}{14.4}{\rmdefault}{\mddefault}{\updefault}{\color[rgb]{0,0,0}$V_1$}%
}}}}
\put(3106,-871){\makebox(0,0)[lb]{\smash{{\SetFigFont{12}{14.4}{\rmdefault}{\mddefault}{\updefault}{\color[rgb]{0,0,0}$V_2$}%
}}}}
\end{picture}%

%% file: len904.pdf_t
\begin{picture}(0,0)%
\includegraphics{len904.pdf}%
\end{picture}%
\setlength{\unitlength}{4144sp}%
\begingroup\makeatletter\ifx\SetFigFont\undefined%
\gdef\SetFigFont#1#2#3#4#5{%
  \reset@font\fontsize{#1}{#2pt}%
  \fontfamily{#3}\fontseries{#4}\fontshape{#5}%
  \selectfont}%
\fi\endgroup%
\begin{picture}(5292,1651)(1921,-2600)
\put(2881,-2536){\makebox(0,0)[lb]{\smash{{\SetFigFont{12}{14.4}{\rmdefault}{\mddefault}{\updefault}{\color[rgb]{0,0,0}$(a)$}%
}}}}
\put(6301,-2536){\makebox(0,0)[lb]{\smash{{\SetFigFont{12}{14.4}{\rmdefault}{\mddefault}{\updefault}{\color[rgb]{0,0,0}$(b)$}%
}}}}
\put(1936,-1456){\makebox(0,0)[lb]{\smash{{\SetFigFont{12}{14.4}{\rmdefault}{\mddefault}{\updefault}{\color[rgb]{0,0,0}$(0,1,0)$}%
}}}}
\put(3241,-1321){\makebox(0,0)[lb]{\smash{{\SetFigFont{12}{14.4}{\rmdefault}{\mddefault}{\updefault}{\color[rgb]{0,0,0}$(5,7,8)$}%
}}}}
\put(6166,-2131){\makebox(0,0)[lb]{\smash{{\SetFigFont{12}{14.4}{\rmdefault}{\mddefault}{\updefault}{\color[rgb]{0,0,0}$(3, 2, 0)$}%
}}}}
\put(3106,-2086){\makebox(0,0)[lb]{\smash{{\SetFigFont{12}{14.4}{\rmdefault}{\mddefault}{\updefault}{\color[rgb]{0,0,0}$(1,0,0)$}%
}}}}
\put(5041,-1501){\makebox(0,0)[lb]{\smash{{\SetFigFont{12}{14.4}{\rmdefault}{\mddefault}{\updefault}{\color[rgb]{0,0,0}$(5,  3,0)$}%
}}}}
\put(6526,-1411){\makebox(0,0)[lb]{\smash{{\SetFigFont{12}{14.4}{\rmdefault}{\mddefault}{\updefault}{\color[rgb]{0,0,0}$(50, 31, 8)$}%
}}}}
\end{picture}%

%% file: len902.pdf_t
\begin{picture}(0,0)%
\includegraphics{len902.pdf}%
\end{picture}%
\setlength{\unitlength}{4144sp}%
\begingroup\makeatletter\ifx\SetFigFont\undefined%
\gdef\SetFigFont#1#2#3#4#5{%
  \reset@font\fontsize{#1}{#2pt}%
  \fontfamily{#3}\fontseries{#4}\fontshape{#5}%
  \selectfont}%
\fi\endgroup%
\begin{picture}(6953,2568)(1156,-3280)
\put(1666,-2626){\makebox(0,0)[lb]{\smash{{\SetFigFont{12}{14.4}{\rmdefault}{\mddefault}{\updefault}{\color[rgb]{0,0,0}$V_0$}%
}}}}
\put(3511,-2626){\makebox(0,0)[lb]{\smash{{\SetFigFont{12}{14.4}{\rmdefault}{\mddefault}{\updefault}{\color[rgb]{0,0,0}$V_1$}%
}}}}
\put(3646,-916){\makebox(0,0)[lb]{\smash{{\SetFigFont{12}{14.4}{\rmdefault}{\mddefault}{\updefault}{\color[rgb]{0,0,0}$V_2$}%
}}}}
\put(1621,-871){\makebox(0,0)[lb]{\smash{{\SetFigFont{12}{14.4}{\rmdefault}{\mddefault}{\updefault}{\color[rgb]{0,0,0}$V_3$}%
}}}}
\put(6121,-2581){\makebox(0,0)[lb]{\smash{{\SetFigFont{12}{14.4}{\rmdefault}{\mddefault}{\updefault}{\color[rgb]{0,0,0}$V_0$}%
}}}}
\put(7651,-2536){\makebox(0,0)[lb]{\smash{{\SetFigFont{12}{14.4}{\rmdefault}{\mddefault}{\updefault}{\color[rgb]{0,0,0}$V_1$}%
}}}}
\put(7516,-1276){\makebox(0,0)[lb]{\smash{{\SetFigFont{12}{14.4}{\rmdefault}{\mddefault}{\updefault}{\color[rgb]{0,0,0}$V_2$}%
}}}}
\put(6031,-916){\makebox(0,0)[lb]{\smash{{\SetFigFont{12}{14.4}{\rmdefault}{\mddefault}{\updefault}{\color[rgb]{0,0,0}$V_3$}%
}}}}
\put(2386,-2626){\makebox(0,0)[lb]{\smash{{\SetFigFont{12}{14.4}{\rmdefault}{\mddefault}{\updefault}{\color[rgb]{0,0,0}$(1, 0)$}%
}}}}
\put(2386,-871){\makebox(0,0)[lb]{\smash{{\SetFigFont{12}{14.4}{\rmdefault}{\mddefault}{\updefault}{\color[rgb]{0,0,0}$(3, 8)$}%
}}}}
\put(1171,-1726){\makebox(0,0)[lb]{\smash{{\SetFigFont{12}{14.4}{\rmdefault}{\mddefault}{\updefault}{\color[rgb]{0,0,0}$(1, 6)$}%
}}}}
\put(3646,-1771){\makebox(0,0)[lb]{\smash{{\SetFigFont{12}{14.4}{\rmdefault}{\mddefault}{\updefault}{\color[rgb]{0,0,0}$(2, -9)$}%
}}}}
\put(5311,-1771){\makebox(0,0)[lb]{\smash{{\SetFigFont{12}{14.4}{\rmdefault}{\mddefault}{\updefault}{\color[rgb]{0,0,0}$(1,0,0)$}%
}}}}
\put(6796,-2761){\makebox(0,0)[lb]{\smash{{\SetFigFont{12}{14.4}{\rmdefault}{\mddefault}{\updefault}{\color[rgb]{0,0,0}$(0,1,0)$}%
}}}}
\put(7111,-871){\makebox(0,0)[lb]{\smash{{\SetFigFont{12}{14.4}{\rmdefault}{\mddefault}{\updefault}{\color[rgb]{0,0,0}$(0,0,1)$}%
}}}}
\put(7021,-3211){\makebox(0,0)[lb]{\smash{{\SetFigFont{12}{14.4}{\rmdefault}{\mddefault}{\updefault}{\color[rgb]{0,0,0}$(b)$}%
}}}}
\put(2566,-3211){\makebox(0,0)[lb]{\smash{{\SetFigFont{12}{14.4}{\rmdefault}{\mddefault}{\updefault}{\color[rgb]{0,0,0}$(a)$}%
}}}}
\put(7921,-1951){\makebox(0,0)[lb]{\smash{{\SetFigFont{12}{14.4}{\rmdefault}{\mddefault}{\updefault}{\color[rgb]{0,0,0}$(-q_1, -q_2, p)$}%
}}}}
\end{picture}%

%% file: len905.pdf_t
\begin{picture}(0,0)%
\includegraphics{len905.pdf}%
\end{picture}%
\setlength{\unitlength}{4144sp}%
\begingroup\makeatletter\ifx\SetFigFont\undefined%
\gdef\SetFigFont#1#2#3#4#5{%
  \reset@font\fontsize{#1}{#2pt}%
  \fontfamily{#3}\fontseries{#4}\fontshape{#5}%
  \selectfont}%
\fi\endgroup%
\begin{picture}(6780,3108)(1156,-2470)
\put(1171,-2131){\makebox(0,0)[lb]{\smash{{\SetFigFont{12}{14.4}{\rmdefault}{\mddefault}{\updefault}{\color[rgb]{0,0,0}$V_1$}%
}}}}
\put(1981,-2131){\makebox(0,0)[lb]{\smash{{\SetFigFont{12}{14.4}{\rmdefault}{\mddefault}{\updefault}{\color[rgb]{0,0,0}$V_2$}%
}}}}
\put(3196,-1411){\makebox(0,0)[lb]{\smash{{\SetFigFont{12}{14.4}{\rmdefault}{\mddefault}{\updefault}{\color[rgb]{0,0,0}$V_i$}%
}}}}
\put(3196,-511){\makebox(0,0)[lb]{\smash{{\SetFigFont{12}{14.4}{\rmdefault}{\mddefault}{\updefault}{\color[rgb]{0,0,0}$V_{i+1}$}%
}}}}
\put(2791,-1726){\makebox(0,0)[lb]{\smash{{\SetFigFont{12}{14.4}{\rmdefault}{\mddefault}{\updefault}{\color[rgb]{0,0,0}$V_3$}%
}}}}
\put(5986,-2086){\makebox(0,0)[lb]{\smash{{\SetFigFont{12}{14.4}{\rmdefault}{\mddefault}{\updefault}{\color[rgb]{0,0,0}$V_1$}%
}}}}
\put(6706,-2041){\makebox(0,0)[lb]{\smash{{\SetFigFont{12}{14.4}{\rmdefault}{\mddefault}{\updefault}{\color[rgb]{0,0,0}$V_2$}%
}}}}
\put(1486,-1816){\makebox(0,0)[lb]{\smash{{\SetFigFont{12}{14.4}{\rmdefault}{\mddefault}{\updefault}{\color[rgb]{0,0,0}$(q_1, p_1)$}%
}}}}
\put(2476,-1951){\makebox(0,0)[lb]{\smash{{\SetFigFont{12}{14.4}{\rmdefault}{\mddefault}{\updefault}{\color[rgb]{0,0,0}$(q_2, p_2)$}%
}}}}
\put(1441, 74){\makebox(0,0)[lb]{\smash{{\SetFigFont{12}{14.4}{\rmdefault}{\mddefault}{\updefault}{\color[rgb]{0,0,0}$(q_k, p_k)$}%
}}}}
\put(3196,-961){\makebox(0,0)[lb]{\smash{{\SetFigFont{12}{14.4}{\rmdefault}{\mddefault}{\updefault}{\color[rgb]{0,0,0}$(q_i, p_i)$}%
}}}}
\put(7921,-1231){\makebox(0,0)[lb]{\smash{{\SetFigFont{12}{14.4}{\rmdefault}{\mddefault}{\updefault}{\color[rgb]{0,0,0}$V_i$}%
}}}}
\put(7921,-421){\makebox(0,0)[lb]{\smash{{\SetFigFont{12}{14.4}{\rmdefault}{\mddefault}{\updefault}{\color[rgb]{0,0,0}$V_{i+1}$}%
}}}}
\put(2746,-2401){\makebox(0,0)[lb]{\smash{{\SetFigFont{12}{14.4}{\rmdefault}{\mddefault}{\updefault}{\color[rgb]{0,0,0}$(a)$}%
}}}}
\put(6976,-2401){\makebox(0,0)[lb]{\smash{{\SetFigFont{12}{14.4}{\rmdefault}{\mddefault}{\updefault}{\color[rgb]{0,0,0}$(b)$}%
}}}}
\put(1171,389){\makebox(0,0)[lb]{\smash{{\SetFigFont{12}{14.4}{\rmdefault}{\mddefault}{\updefault}{\color[rgb]{0,0,0}$V_{k+1}$}%
}}}}
\put(2071,344){\makebox(0,0)[lb]{\smash{{\SetFigFont{12}{14.4}{\rmdefault}{\mddefault}{\updefault}{\color[rgb]{0,0,0}$V_{k}$}%
}}}}
\put(2746,-61){\makebox(0,0)[lb]{\smash{{\SetFigFont{12}{14.4}{\rmdefault}{\mddefault}{\updefault}{\color[rgb]{0,0,0}$V_{k-1}$}%
}}}}
\put(6841,389){\makebox(0,0)[lb]{\smash{{\SetFigFont{12}{14.4}{\rmdefault}{\mddefault}{\updefault}{\color[rgb]{0,0,0}$V_{k}$}%
}}}}
\put(5941,479){\makebox(0,0)[lb]{\smash{{\SetFigFont{12}{14.4}{\rmdefault}{\mddefault}{\updefault}{\color[rgb]{0,0,0}$V_{k+1}$}%
}}}}
\end{picture}%

%% file: len903.pdf_t
\begin{picture}(0,0)%
\includegraphics{len903.pdf}%
\end{picture}%
\setlength{\unitlength}{4144sp}%
\begingroup\makeatletter\ifx\SetFigFont\undefined%
\gdef\SetFigFont#1#2#3#4#5{%
  \reset@font\fontsize{#1}{#2pt}%
  \fontfamily{#3}\fontseries{#4}\fontshape{#5}%
  \selectfont}%
\fi\endgroup%
\begin{picture}(7763,4003)(-1904,-4670)
\put(5536,-2671){\makebox(0,0)[lb]{\smash{{\SetFigFont{12}{14.4}{\rmdefault}{\mddefault}{\updefault}{\color[rgb]{0,0,0}$(-q_1, -q_2, p)$}%
}}}}
\put(1396,-2356){\makebox(0,0)[lb]{\smash{{\SetFigFont{12}{14.4}{\rmdefault}{\mddefault}{\updefault}{\color[rgb]{0,0,0}  $(0,1,0)$}%
}}}}
\put(2926,-3931){\makebox(0,0)[lb]{\smash{{\SetFigFont{12}{14.4}{\rmdefault}{\mddefault}{\updefault}{\color[rgb]{0,0,0}$V_{0_1}$}%
}}}}
\put(4141,-3931){\makebox(0,0)[lb]{\smash{{\SetFigFont{12}{14.4}{\rmdefault}{\mddefault}{\updefault}{\color[rgb]{0,0,0}$V_{1_0}$}%
}}}}
\put(3061,-3121){\makebox(0,0)[lb]{\smash{{\SetFigFont{12}{14.4}{\rmdefault}{\mddefault}{\updefault}{\color[rgb]{0,0,0}$V_{0_2}$}%
}}}}
\put(3691,-826){\makebox(0,0)[lb]{\smash{{\SetFigFont{12}{14.4}{\rmdefault}{\mddefault}{\updefault}{\color[rgb]{0,0,0}$1, 0, 0)$}%
}}}}
\put(-1034,-3931){\makebox(0,0)[lb]{\smash{{\SetFigFont{12}{14.4}{\rmdefault}{\mddefault}{\updefault}{\color[rgb]{0,0,0}$V_0$}%
}}}}
\put(1261,-3886){\makebox(0,0)[lb]{\smash{{\SetFigFont{12}{14.4}{\rmdefault}{\mddefault}{\updefault}{\color[rgb]{0,0,0}$V_1$}%
}}}}
\put(856,-1996){\makebox(0,0)[lb]{\smash{{\SetFigFont{12}{14.4}{\rmdefault}{\mddefault}{\updefault}{\color[rgb]{0,0,0}$V_2$}%
}}}}
\put(-989,-1321){\makebox(0,0)[lb]{\smash{{\SetFigFont{12}{14.4}{\rmdefault}{\mddefault}{\updefault}{\color[rgb]{0,0,0}$V_3$}%
}}}}
\put(-89,-1276){\makebox(0,0)[lb]{\smash{{\SetFigFont{12}{14.4}{\rmdefault}{\mddefault}{\updefault}{\color[rgb]{0,0,0}$(-q_1, -q_2, p)$}%
}}}}
\put(-1889,-1951){\makebox(0,0)[lb]{\smash{{\SetFigFont{12}{14.4}{\rmdefault}{\mddefault}{\updefault}{\color[rgb]{0,0,0}$(1, 0, 0)$}%
}}}}
\put(-1889,-2806){\makebox(0,0)[lb]{\smash{{\SetFigFont{12}{14.4}{\rmdefault}{\mddefault}{\updefault}{\color[rgb]{0,0,0}$(0, 1, 0)$}%
}}}}
\put( 46,-4561){\makebox(0,0)[lb]{\smash{{\SetFigFont{12}{14.4}{\rmdefault}{\mddefault}{\updefault}{\color[rgb]{0,0,0}(1)}%
}}}}
\put(3736,-4606){\makebox(0,0)[lb]{\smash{{\SetFigFont{12}{14.4}{\rmdefault}{\mddefault}{\updefault}{\color[rgb]{0,0,0}(2)}%
}}}}
\put( 91,-4201){\makebox(0,0)[lb]{\smash{{\SetFigFont{12}{14.4}{\rmdefault}{\mddefault}{\updefault}{\color[rgb]{0,0,0}$(-a, -b, c)$}%
}}}}
\put(3466,-4291){\makebox(0,0)[lb]{\smash{{\SetFigFont{12}{14.4}{\rmdefault}{\mddefault}{\updefault}{\color[rgb]{0,0,0}$(-a, -b, c)$}%
}}}}
\put(1891,-2941){\makebox(0,0)[lb]{\smash{{\SetFigFont{12}{14.4}{\rmdefault}{\mddefault}{\updefault}{\color[rgb]{0,0,0}$V_{0_3}$}%
}}}}
\put(5401,-2986){\makebox(0,0)[lb]{\smash{{\SetFigFont{12}{14.4}{\rmdefault}{\mddefault}{\updefault}{\color[rgb]{0,0,0}$V_{1_2}$}%
}}}}
\put(4231,-2671){\makebox(0,0)[lb]{\smash{{\SetFigFont{12}{14.4}{\rmdefault}{\mddefault}{\updefault}{\color[rgb]{0,0,0}$V_{1_3}$}%
}}}}
\put(5311,-1951){\makebox(0,0)[lb]{\smash{{\SetFigFont{12}{14.4}{\rmdefault}{\mddefault}{\updefault}{\color[rgb]{0,0,0}$V_{2_1}$}%
}}}}
\put(4501,-2041){\makebox(0,0)[lb]{\smash{{\SetFigFont{12}{14.4}{\rmdefault}{\mddefault}{\updefault}{\color[rgb]{0,0,0}$V_{2_0}$}%
}}}}
\put(4321,-1231){\makebox(0,0)[lb]{\smash{{\SetFigFont{12}{14.4}{\rmdefault}{\mddefault}{\updefault}{\color[rgb]{0,0,0}$V_{2_3}$}%
}}}}
\put(2836,-961){\makebox(0,0)[lb]{\smash{{\SetFigFont{12}{14.4}{\rmdefault}{\mddefault}{\updefault}{\color[rgb]{0,0,0}$V_{3_2}$}%
}}}}
\put(3061,-1636){\makebox(0,0)[lb]{\smash{{\SetFigFont{12}{14.4}{\rmdefault}{\mddefault}{\updefault}{\color[rgb]{0,0,0}$V_{3_1}$}%
}}}}
\put(1891,-1636){\makebox(0,0)[lb]{\smash{{\SetFigFont{12}{14.4}{\rmdefault}{\mddefault}{\updefault}{\color[rgb]{0,0,0}$V_{3_0}$}%
}}}}
\end{picture}%